%% file: MobiusFunctionOfIntervalsWithAnIndecomposableLowerBound.tex
\documentclass[a4paper]{article}
\usepackage{amsmath}
\usepackage{amsthm}
\usepackage{amssymb}
\usepackage{titlefoot} 
\usepackage{tikz}
\usepackage{tkz-euclide}
\usepackage{mathtools} 
\usepackage{subcaption}
\usepackage{stmaryrd}
\usepackage{booktabs}
\usepackage{graphicx}
\usepackage{pgfplots}
\usetikzlibrary{plotmarks}
\pgfplotsset{compat=newest}

\theoremstyle{plain}
\newtheorem{theorem}{Theorem}
\newtheorem{lemma}[theorem]{Lemma}
\newtheorem{corollary}[theorem]{Corollary}
\newtheorem{proposition}[theorem]{Proposition}

\newtheorem{observation}[theorem]{Observation}

\theoremstyle{definition}

\newtheorem{conjecture}[theorem]{Conjecture}

\theoremstyle{remark}

\newcommand{\tinydot}[1]{
    \node at #1 {\tiny $\bullet$};
}
\newcommand{\evensmallerdot}[1]{
    \node at #1 {\scriptsize $\bullet$};
}
\newcommand{\smallerdot}[1]{
    \node at #1 {\footnotesize $\bullet$};
}
\newcommand{\smalldot}[1]{
    \node at #1 {\small $\bullet$};
}
\newcommand{\orderdot}[2]{
    \node at #1 {\normalsize $\bullet$};
    \node at #1 [below] {$#2$};
}
\newcommand{\opendot}[1]{
    \node at #1 {\normalsize \textcolor{white}{$\bullet$}};
    \node at #1 {\normalsize $\circ$};
}
\newcommand{\plotgrid}[2]{
    \foreach \i in {0,1,...,#2}{
        \draw [color=lightgray] (0.5, {\i+0.5})--({#1+0.5}, {\i+0.5});
    };
    \foreach \i in {0,1,...,#1}{
        \draw [color=lightgray] ({\i+0.5}, 0.5)--({\i+0.5}, {#2+0.5});
    };
}
\newcommand{\plotperm}[1]{
    \foreach \j [count=\i] in {#1} {
        \orderdot{(\i,\j)}{};
    };
}
\newcommand{\darkhline}[2]{
    \draw [color=darkgray, thick]
    (0.5,{#1+0.5}) -- ({#2+0.5},{#1+0.5});
}
\newcommand{\darkvline}[2]{
    \draw [color=darkgray, thick]
    ({#1+0.5},0.5) -- ({#1+0.5},{#2+0.5});
}
\newcommand{\plotpermgrid}[1]{
    \foreach \i [count=\n] in {#1} {};
    \foreach \i in {0,1,2,...,\n}{
        \draw [color=lightgray] ({\i+0.5}, 0.5)--({\i+0.5}, {\n+0.5});
        \draw [color=lightgray] (0.5, {\i+0.5})--({\n+0.5}, {\i+0.5});
    }
    \plotperm{#1};
}

\newcommand{\mob}{M\"{o}bius }
\newcommand{\order}[1]{\ensuremath{\left\lvert#1\right\rvert}}
\renewcommand{\interleave}{\obslash} 
\newcommand{\skewinterleave}{\oslash}
\newcommand{\mobfn}[2]{\mu[#1,#2]}
\newcommand{\oneplus}{1 \oplus} 
\newcommand{\plusone}{\oplus 1}
\newcommand{\oneplusone}{1 \oplus 1}
\newcommand{\oneil}{1 \interleave}
\newcommand{\ilone}{\interleave 1}
\newcommand{\familysum}[1]{\mathcal{F}_{\oplus}(#1)}
\newcommand{\familyil}[1]{\mathcal{F}_{\interleave}(#1)}
\newcommand{\contrib}[2]{\mathfrak{C}_{#1,#2}}
\newcommand{\shapes}{\mathcal{S}}
\newcommand{\weightgen}[3]{W(#1,#2,#3)}
\newcommand{\weightosc}[3]{W_{io}(#1,#2,#3)}
\newcommand{\dtwo}{21}
\newcommand{\nsums}[2]{\oplus^{#1}#2}
\newcommand{\nils}[2]{\interleave^{#1}#2}
\newcommand{\sumra}{\nsums{r}{\alpha}}
\newcommand{\sumrab}{\left(\sumra\right)}
\newcommand{\ildtwok}{\nils{k}{\dtwo}}
\newcommand{\ildtwokb}{\left(\ildtwok\right)}
\newcommand{\ildtwokp}{\nils{{k+1}}{\dtwo}}
\newcommand{\ildtwokpb}{\left(\ildtwokp\right)}

\DeclareMathOperator{\mink}{MinK}
\DeclareMathOperator{\maxk}{MaxK}
\DeclareMathOperator{\rawmink}{RawMinK}

\makeatletter
\def\anonfootnote{\xdef\@thefnmark{}\@footnotetext}
\makeatother

\makeatletter
\def\mymod#1{\allowbreak\mkern10mu({\operator@font mod}\,\,#1)}
\makeatother

\title{The \mob function of permutations with an indecomposable lower bound}
\author{
        Robert Brignall, 
        David Marchant
        \\
        \\
        \textit{School of Mathematics and Statistics}\\[-3pt]
        \textit{The Open University, Milton Keynes, MK7 6AA, UK}\\[10pt]
}
\amssubj{05A05} 

\begin{document}
\maketitle
\anonfootnote{\textit{Email:} \{robert.brignall; david.marchant\}@open.ac.uk}
\begin{abstract}    
    We show that the \mob function
    of an interval in a permutation poset
    where the lower bound is sum (resp. skew) indecomposable depends solely
    on the sum (resp. skew) indecomposable permutations contained
    in the upper bound,
    and that this can simplify the calculation
    of the \mob sum.
    For increasing oscillations,
    we give a recursion for the \mob sum
    which only involves evaluating
    simple inequalities.
\end{abstract}

\section{Introduction}

Let $\sigma$ and $\pi$ be permutations of natural numbers,
written in one-line notation, 
with 
$\sigma = \sigma_1 \sigma_2 \ldots \sigma_m$,
and
$\pi = \pi_1 \pi_2 \ldots \pi_n$.
We say that $\sigma$ is \emph{contained} in $\pi$ 
if there is a sequence  
$1 \leq i_1 < i_2 < \ \ldots < i_m \leq n$
such that for any $r,s \in \{1,\ldots,m\}$,
$\pi_{i_r} < \pi_{i_s}$ if and only if $\sigma_r < \sigma_s$.  
We say that $\pi$ \emph{avoids} $\sigma$ if $\pi$ does not contain $\sigma$.
The set of all permutations is a poset under the partial order given by
containment.  

%
%
A closed interval $[\sigma, \pi]$ in a poset is the set
defined as $\{ \tau : \sigma \leq \tau \leq \pi \}$.
A half-open interval $[\sigma, \pi)$ is the set
$\{ \tau : \sigma \leq \tau < \pi \}$.
The \mob function $\mobfn{\sigma}{\pi}$ 
is defined on an interval of a poset as follows:
for $\sigma \nleq \pi$, $\mobfn{\sigma}{\pi} = 0$;
for all $\lambda$, $\mobfn{\lambda}{\lambda} = 1$;
and for $\sigma < \pi$, 
\begin{align}
\mobfn{\sigma}{\pi} &
= 
- \sum_{\lambda \in [\sigma, \pi)} \mobfn{\sigma}{\lambda}.
\label{equation_mobius_function}
\end{align}

%
%
Our motivation for this paper is to find 
a \emph{contributing set} $\contrib{\sigma}{\pi}$
that is significantly smaller than
the poset interval $[\sigma, \pi)$,
and a $\{0, \pm1 \}$ weighting function $\weightgen{\sigma}{\alpha}{\pi}$
such that
\begin{align}
\label{equation_contributing_sum}
\mobfn{\sigma}{\pi} 
& = 
- \sum_{\alpha \in \contrib{\sigma}{\pi}}
\mobfn{\sigma}{\alpha} \weightgen{\sigma}{\alpha}{\pi}.
\end{align}
Plainly, in Equation~\ref{equation_contributing_sum},
we could set 
$\contrib{\sigma}{\pi} = [\sigma, \pi)$,
and $\weightgen{\sigma}{\alpha}{\pi} = 1$,
which is equivalent to
Equation~\ref{equation_mobius_function}.

One approach here would be to
take a permutation $\beta$ 
such that
$\sigma < \beta < \pi$.
We could then set 
$\contrib{\sigma}{\pi} 
= 
\{
\lambda : 
\lambda \in [\sigma, \pi) 
\text{ and } 
\lambda \not\in [\sigma, \beta]
\}$,
and $\weightgen{\sigma}{\alpha}{\pi} = 1$,
since, from Equation~\ref{equation_contributing_sum},
$\sum_{\lambda \in [\sigma, \beta]} \mobfn{\sigma}{\lambda} = 0$.
This approach was used in 
Smith~\cite{Smith2013},
who determined the \mob function on the interval
$[1, \pi]$ for all permutations $\pi$ with a single descent.
Smith's paper is unusual, in that it provides an explicit formula
for the value of the \mob function.

Our approach is different.  
We identify individual elements (say $\lambda$), of the poset
that have $\mobfn{\sigma}{\lambda} = 0$.  
We also
show that there are pairs of elements, 
$\lambda$ and $\lambda^\prime$,
where $\mobfn{\sigma}{\lambda} = - \mobfn{\sigma}{\lambda^\prime}$,
and so we can exclude these pairs of elements.
Finally, we show that there are
quartets of permutations $\lambda_1, \ldots, \lambda_4$
where $\sum_{i=1}^{4} \mobfn{\sigma}{\lambda_i} = 0$;
and that we can systematically identify these quartets.
By excluding these permutations from
$\contrib{\sigma}{\pi}$
we can significantly reduce the number of elements
in $\contrib{\sigma}{\pi}$ 
compared to the number of elements in the
interval $[\sigma, \pi)$. 
This approach results in
the ability to compute 
$\mobfn{\sigma}{\pi}$, 
where
$\sigma$ is indecomposable,
much faster than evaluating 
Equation~\ref{equation_mobius_function}.
For increasing oscillations, we will show
that the elements of $\contrib{\sigma}{\pi}$
can be determined using simple inequalities,
and that as a consequence
$\mobfn{\sigma}{\pi}$ can be determined
using inequalities.
With this approach, we 
have computed $\mobfn{1}{\pi}$,
where $\pi$ is an increasing oscillation,
up to $\order{\pi} = \text{2,000,000}$. 

%
%
The study of the \mob function in the permutation poset
was introduced by
Wilf~\cite{Wilf2002}.  
The first result in this area was by Sagan and Vatter~\cite{Sagan2006}, 
who determined the \mob function 
on intervals of layered permutations.
Steingr\'{\i}msson and Tenner~\cite{Steingrimsson2010} found a 
large class of pairs of permutations $(\sigma, \pi)$ 
where $\mobfn{\sigma}{\pi} = 0$, as well as determining the 
\mob function where $\sigma$ occurs exactly once in $\pi$, 
and $\sigma$ and $\pi$ satisfy certain other conditions.
Burstein, Jel{\'{i}}nek, Jel{\'{i}}nkov{\'{a}} 
and Steingr{\'{i}}msson~\cite{Burstein2011} found
a recursion for the \mob function
for sum/skew decomposable permutations
in terms of the sum/skew indecomposable 
permutations in the lower and upper bounds.
They also found a method to determine
the \mob function for separable permutations
by counting embeddings.
We use the recursions for 
decomposable permutations to underpin the first
part of this paper.
McNamara and Steingr{\'{i}}msson~\cite{McNamara2015}
investigated the topology of intervals in the 
permutation poset, and found a single recurrence
equivalent to the recursions in~\cite{Burstein2011}.

Many results for $\mobfn{\sigma}{\pi}$ 
have been obtained by considering
ways in which $\sigma$ can be found in $\pi$,
which we call an \emph{embedding} of $\sigma$ in $\pi$,
although typically only some of the embeddings 
(``normal embeddings'') are counted.
In \cite{Smith2016}, 
Smith used normal embeddings to
determine the \mob function $\mobfn{\sigma}{\pi}$
when $\sigma$ and $\pi$ have the same number of descents.

One problem with embeddings
arises in cases such as $\mobfn{1}{24153}$.
Here there are plainly only five ways
to embed the permutation $1$ 
into $24153$, however
$\mobfn{1}{24153} = 6$,
and thus the embedding approach
is not sufficient.

One possible solution to this issue
is to count the normal embeddings and then
add a ``correction factor''.
This approach
was used by Smith in~\cite{Smith2016a}.
The result applies to all intervals in the 
permutation poset, although the
correction factor is a rather complicated
double sum, which still involves
$\sum_{\lambda \in [\sigma, \pi)} \mobfn{\sigma}{\lambda}$.

In this paper we
show that the \mob function
on intervals with a sum indecomposable lower bound
depends only on the sum indecomposable permutations
contained in the upper bound.
We provide a weighting function that
determines which sum indecomposable permutations
contribute to the \mob sum.
We then consider increasing oscillations.
For these permutations, we show how we can
find all of the permutations that contribute to the \mob sum
by applying simple numeric inequalities,
which leads to a fast polynomial algorithm
for determining the \mob function.

We start with some essential definitions and notation
in Section~\ref{section-definitions-and-notation},
then in Section~\ref{section-preliminary-lemmas} 
we provide a number of preliminary lemmas.
We conclude this section with a theorem
that gives $\mobfn{\sigma}{\pi}$,
where $\sigma$
is a sum indecomposable permutation,
for all $\pi$.
In Section~\ref{section-increasing-oscillations}
we consider $\mobfn{\sigma}{\pi}$
where $\sigma$
is a sum indecomposable permutation,
and $\pi$ is an increasing oscillation.
We finish with some concluding remarks in 
Section~\ref{section-concluding-remarks}.

\section{Definitions and notation}\label{section-definitions-and-notation}

When discussing the \mob function, 
$\mobfn{\sigma}{\pi} = - \sum_{\lambda \in [\sigma, \pi)} \mobfn{\sigma}{\lambda}$,
we will frequently be examining the value of
$\mobfn{\sigma}{\lambda}$ for a specific permutation 
$\lambda$.  We say that this is the \emph{contribution} that 
$\lambda$ makes to the sum.  If we have a set of permutations
$S \subseteq [\sigma, \pi)$
such that
$\sum_{\lambda \in S} \mobfn{\sigma}{\lambda} = 0$, 
then we say that the set $S$ makes
\emph{no net contribution} to the sum.

A \emph{sum} 
of two permutations $\alpha$ and $\beta$
of lengths $m$ and $n$ respectively
is the permutation 
$
\alpha_1, \ldots, \alpha_m,
\beta_1 + m, \ldots, \beta_n + m
$.
We write a sum as $\alpha \oplus \beta$.
A \emph{skew sum}, $\alpha \ominus \beta$,
is the permutation
$
\alpha_1 +n, \ldots, \alpha_m + n,
\beta_1, \ldots, \beta_n
$.
As examples,
$321 \oplus 213 = 321546$, 
and 
$321 \ominus 213 = 654213$, 
and these are shown in 
Figure~\ref{figure_examples-of-sums-and-interleaves}.
A \emph{sum-indecomposable} 
(resp. \emph{skew-indecomposable})
permutation is a permutation
that cannot be written as the 
sum (resp. skew-sum) of two smaller permutations.

Let $\alpha$ be a permutation, 
and $r$ a positive integer.
Then $\sumra$
is $\alpha \oplus \alpha \oplus \ldots \oplus \alpha \oplus \alpha$,
with $r$ occurrences of $\alpha$.
If $S$ is a set of permutations, then
$\nsums{r}{S} = \cup_{\lambda \in S} \{ \nsums{r}{\lambda} \}$.

The \emph{interleave} of two permutations
$\alpha$ and $\beta$ is formed by taking
the sum $\alpha \oplus \beta$,
and then exchanging the value of the largest point from
$\alpha$ with the value of the smallest point from $\beta$.
We can also view this as increasing the largest
point from $\alpha$ by 1, and simultaneously
decreasing the smallest point from $\beta$ by 1.
We write an interleave as $\alpha \interleave \beta$.
For example, $321 \interleave 213 = 421536$,
see Figure~\ref{figure_examples-of-sums-and-interleaves}.

For completeness, we also define
a \emph{skew interleave}, $\alpha \skewinterleave \beta$,
which is formed 
by taking the skew sum $\alpha \ominus \beta$,
and then exchanging the smallest point from
$\alpha$ with the largest point from $\beta$.
As an example,
$321 \skewinterleave 213 = 653214$,
as shown in 
Figure~\ref{figure_examples-of-sums-and-interleaves}.

The interleave operations, 
$\interleave$ and $\skewinterleave$,
are not associative, as
$\oneil \oneil 1$ could represent
$231$ or $312$.  To avoid this ambiguity,
we require that the permutation $1$
can either be interleaved to the left
or to the right, but not both.  
It is easy to see that this restriction
establishes associativity.
We note here that, with this restriction,
an expression involving $\oplus$ and $\interleave$
represents a unique permutation
regardless of the order in which the operations
are applied.

Let $\alpha$ be a permutation with length greater than 1.  
We will frequently want to refer to 
permutations that 
have the form
$\alpha \interleave \alpha \interleave 
\ldots
\interleave \alpha \interleave \alpha$.
If there are $n$ copies of $\alpha$ being
interleaved, then we will write this as
$\nils{n}{\alpha}$, so, for example, 
we have $\nils{3}{(21)} = 21 \interleave 21 \interleave 21 = 315264$.

\begin{figure}[ht]
    \begin{footnotesize}
        \begin{center}
            \begin{subfigure}[t]{0.22\textwidth}
                \begin{center}
                    \begin{tikzpicture}[scale=0.25]   
                    \plotpermgrid{3,2,1,5,4,6}
                    \end{tikzpicture}
                \end{center}
                \caption*{$321 \oplus 213$}
            \end{subfigure}
            \begin{subfigure}[t]{0.22\textwidth}
                \begin{center}
                    \begin{tikzpicture}[scale=0.25]                    
                    \plotpermgrid{6,5,4,2,1,3}
                    \end{tikzpicture}
                \end{center}
                \caption*{$321 \ominus 213$}
            \end{subfigure}
            \begin{subfigure}[t]{0.22\textwidth}
                \begin{center}
                    \begin{tikzpicture}[scale=0.25]                    
                    \plotpermgrid{4,2,1,5,3,6}
                    \end{tikzpicture}
                \end{center}
                \caption*{$321 \interleave 213$}
            \end{subfigure}
            \begin{subfigure}[t]{0.22\textwidth}
                \begin{center}
                    \begin{tikzpicture}[scale=0.25]                    
                    \plotpermgrid{6,5,3,2,1,4}
                    \end{tikzpicture}
                \end{center}
                \caption*{$321 \skewinterleave 213$}
            \end{subfigure}
        \end{center}
    \end{footnotesize}
    \caption{Examples of direct and skew sums and interleaves.} 
    \label{figure_examples-of-sums-and-interleaves}	
\end{figure}
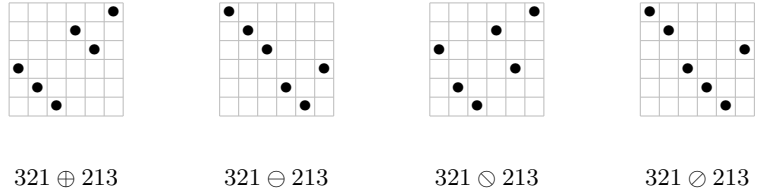    

For the remainder of this paper, 
by symmetry it suffices to 
discuss permutations in relation to 
sums and interleaves only.
For the same reason, 
references to (in)decomposable permutations 
may omit the ``sum'' qualifier.

An \emph{interval} of a permutation 
is a set of contiguous positions 
where the set of values is also contiguous.
A \emph{simple} permutation is one where
the permutation contains no intervals
other than those of length one, and the entire permutation.

The \emph{increasing oscillating sequence} 
is the sequence
\[
4, 1, 6, 3, 8, 5, 10, 7, \ldots, 2k+2, 2k-1, \ldots.
\]
The start of the sequence is depicted in 
Figure~\ref{figure_increasing_oscillating_sequence}.  
\begin{figure}
    \centering
    \begin{subfigure}[c]{0.4\textwidth}
        \centering
        \begin{tikzpicture}[scale=0.2]
        \plotgrid{12}{14}
        \plotperm{4, 1, 6, 3, 8, 5}
        \smalldot{(7,10)}  
        \smalldot{(8,7)}
        \smallerdot{(9, 12)}
        \smallerdot{(10,  9)}
        \evensmallerdot{(11, 14)}
        \evensmallerdot{(12, 11)}
        \tinydot{(13, 16)}
        \tinydot{(14, 13)}
        \end{tikzpicture}
    \end{subfigure}
    \caption{A depiction of the start of the increasing oscillating sequence.}
    \label{figure_increasing_oscillating_sequence}
\end{figure}
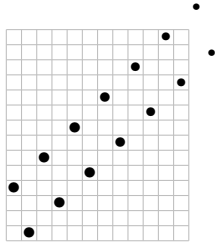  
An \emph{increasing oscillation}
is a simple permutation contained in
the increasing oscillating sequence.
For lengths greater than three, 
there are exactly two increasing oscillations
of each length.  
Let $W_n$ be the increasing oscillation 
with $n$ elements which starts with a descent,
and
let $M_n$ be the increasing oscillation
with $n$ elements which starts with an ascent.  Then
\begin{align*}
W_{2n}   & = \nils{n}{\dtwo},  &
M_{2n}   & = \oneil \left( \nils{n-1}{\dtwo} \right) \ilone,  \\
W_{2n-1} & = \left( \nils{n-1}{\dtwo} \right) \ilone, \qquad \text{and} &
M_{2n-1} & = \oneil \left( \nils{n-1}{\dtwo} \right). \\
\end{align*}
Note that $W_n = M_n^{-1}$.

There are instances where,
for some permutation $\alpha$, 
we are interested in the set of
permutations 
$
\{\alpha, \;
\oneplus \alpha, \;
\alpha \plusone, \;
\oneplus \alpha \plusone
\}$.
Given a permutation $\alpha$, we
refer to this set 
as $\familysum{\alpha}$,
and we 
say that this set 
is the \emph{family} of $\alpha$.
If $S$ is a set of permutations, then
$\familysum{S}= \cup_{\alpha \in S} \{\familysum{\alpha} \}$.

There are some also instances where we 
are interested in the set of
permutations 
$
\familyil{\alpha} = 
\{\alpha, \;
\oneil \alpha, \;
\alpha \ilone, \;
\oneil \alpha \ilone
\}$.
Note that every increasing oscillation
is an element of 
$
\familyil{\ildtwok}
$
for some $k \geq 1$.

\section{Preliminary lemmas and main theorem}\label{section-preliminary-lemmas}

In this section our aim is to show that 
if $\sigma$ is indecomposable, then for any $\pi \geq \sigma$
there is a $\{0, \pm 1\}$ weighting function
$\weightgen{\sigma}{\alpha}{\pi}$ 
and a set of permutations
$\contrib{\sigma}{\pi}$, such that
\[
\mobfn{\sigma}{\pi} 
= 
- \sum_{\alpha \in \contrib{\sigma}{\pi}}
\mobfn{\sigma}{\alpha} \weightgen{\sigma}{\alpha}{\pi}.
\]

If $\pi$ is 
the identity permutation
$12 \ldots n$ or its reverse,
then $\mobfn{\sigma}{\pi}$ is trivial for any $\sigma$, and we
exclude the identity and its reverse from
being the upper bound of any interval under consideration.

As noted earlier, our approach is to show that 
there are permutations, pairs of permutations, and
quartets of permutations in 
$[\sigma, \pi)$ 
that make no net contribution to the sum.

We use Proposition 1 and 2, and Corollary 3 
from
Burstein, 
Jel{\'{i}}nek, 
Jel{\'{i}}nkov{\'{a}} and 
Steingr{\'{i}}msson~\cite{Burstein2011}.
We start with some required notation.
If $\pi$ is a non-empty permutation with decomposition
$\pi_\oneplus \ldots \oplus \pi_n$, then 
for any integer $i$ with $0 \leq i \leq n$, 
$\pi_{\leq i}$ is the permutation 
$\pi_\oneplus \ldots \oplus \pi_i$,
and 
$\pi_{> i}$ is the permutation
$\pi_{i+1} \oplus \ldots \oplus \pi_n$.
An empty sum of permutations is defined as $\varepsilon$, 
and in particular $\pi_{\leq 0} = \pi_{> n} = \varepsilon$.
We can see that 
$\mobfn{\varepsilon}{\varepsilon} = 1$,
$\mobfn{\varepsilon}{1} = -1$ and
$\mobfn{\varepsilon}{\tau} = 0$ for any $\tau > 1$.
We now recall the results from 
Burstein, 
Jel{\'{i}}nek, 
Jel{\'{i}}nkov{\'{a}} and 
Steingr{\'{i}}msson:
\begin{proposition}[{%
    Burstein, 
    Jel{\'{i}}nek, 
    Jel{\'{i}}nkov{\'{a}} and 
    Steingr{\'{i}}msson 
    \cite[Proposition 1]{Burstein2011}}]
    \label{BJJS_proposition_1}
    Let $\sigma$ and $\pi$ be non-empty permutations
    with decompositions 
    $\sigma = \sigma_\oneplus \ldots \oplus \sigma_m$
    and
    $\pi = \pi_\oneplus \ldots \oplus \pi_n$,
    with $n \geq 2$.
    Assume that $\pi_1 = 1$, 
    and let $k$ be the largest integer such that
    $\pi_1, \pi_2, \ldots ,\pi_k$ are all equal to $1$.
    Let $l \geq 0$ be the largest integer such that
    $\sigma_1, \sigma_2, \ldots, \sigma_l$ are all equal to $1$.
    Then
    \begin{align*}
    \mobfn{\sigma}{\pi} & =
    \begin{cases*}
    0 & \text{if $k-1 > l$,} \\
    -\mobfn{\sigma_{> k-1}}{\pi_{> k}} & \text{if $k-1 = l$,} \\
    \mobfn{\sigma_{> k}}{\pi_{> k}} - \mobfn{\sigma_{> k-1}}{\pi_{> k}} & \text{if $k-1 < l$.}
    \end{cases*}
    \end{align*}
\end{proposition}
\begin{proposition}[{%
    \cite[Proposition 2]{Burstein2011}}]
    \label{BJJS_proposition_2}
    Let $\sigma$ and $\pi$ be non-empty permutations
    with decompositions 
    $\sigma = \sigma_\oneplus \ldots \oplus \sigma_m$
    and
    $\pi = \pi_\oneplus \ldots \oplus \pi_n$,
    with $n \geq 2$.
    Assume that $\pi_1 \neq 1$, 
    and let $k$ be the largest integer such that
    $\pi_1, \pi_2, \ldots ,\pi_k$ are all equal to $\pi_1$.
    Then
    \begin{align*}
    \mobfn{\sigma}{\pi} & =
    \sum_{i=1}^m
    \sum_{j=1}^k
    \mobfn{\sigma_{\leq i}}{\pi_{1}}
    \mobfn{\sigma_{> i}}{\pi_{> j}}.
    \end{align*}    
\end{proposition}
\begin{corollary}[{%
    \cite[Corollary 3]{Burstein2011}}]
    \label{BJJS_corollary_3}
    Let $\sigma$ and $\pi$ be as in~\ref{BJJS_proposition_2}.
    Suppose that $\sigma$ is sum indecomposable, so $m = 1$.
    Then
    \begin{align*}
    \mobfn{\sigma}{\pi} & =
    \begin{cases*}
    \mobfn{\sigma}{\pi_1} & \text{if $\pi = \nsums{k}{\pi_1$},} \\
    -\mobfn{\sigma}{\pi_1} & \text{if $\pi = \left(\nsums{k}{\pi_1} \right) \plusone$,} \\
    0 & \text{otherwise,}
    \end{cases*}
    \end{align*}
\end{corollary}
A simple consequence of 
Propositions~\ref{BJJS_proposition_1}
and~\ref{BJJS_proposition_2}
is the identification
of some intervals of permutations where 
the value of the \mob function is zero.
\begin{lemma}
    \label{lemma_mobius_function_is_zero}
    Let $\pi \in 
    \{
    \oneplus \oneplus \tau,
    \tau \plusone \plusone,
    \familysum{\sumrab \oplus \tau^\prime}
    \}
    $,
    where $\tau$ is any permutation,
    $r$ is maximal,
    $\alpha$ is sum indecomposable,
    and
    $\tau^\prime$ is any permutation
    greater than $1$.
    Let $\sigma$ be a sum indecomposable permutation.
    Then
    $
    \mobfn{\sigma}{\pi}  = 0
    $.
\end{lemma}
\begin{proof}
    Consider $\pi = \oneplus \oneplus \tau$.
    We use Proposition~\ref{BJJS_proposition_1}.  
    If $\tau_1 = 1$,
    then $k \geq 3$, and $l \leq 1$, 
    and the result follows immediately.
    Now assume that $\tau_1 \neq 1$.
    Then $k=2$.  
    If $\sigma > 1$, then again the
    result follows immediately.
    If $\sigma=1$, then we have
    $\mobfn{\sigma}{\pi} 
    = - \mobfn{\sigma_{> k-1}}{\pi_{>k}} 
    = - \mobfn{\varepsilon}{\tau} 
    = 0$.
    The case for $\pi = \tau \plusone \plusone$
    follows by symmetry.
    
    Now consider 
    $\pi = \familysum{\sumrab \oplus \tau^\prime}$.    
    If $\pi = \sumrab \oplus \tau$,
    or $\pi = \sumrab \oplus \tau \plusone$,
    then
    we use Proposition~\ref{BJJS_proposition_2}.
    In that context we have $m = 1$ and $k = r$,
    and so
    $
    \mobfn{\sigma}{\pi}
    =
    \sum_{j=1}^r
    \mobfn{\sigma}{\pi_1}
    \mobfn{\varepsilon}{\pi_{> j}} 
    $
    For every value of $j$, $\pi_{> j}$ 
    is non-empty and greater than $1$, and so
    $\mobfn{\varepsilon}{\pi_{> j}} = 0$ for all $j$,
    and hence every term in the sum is zero.     
    If $\pi = \oneplus \sumrab \oplus \tau$
    or
    $\pi = \oneplus \sumrab \oplus \tau \plusone$,
    then
    we use Proposition~\ref{BJJS_proposition_1},
    which reduces to one of the previous cases.    
\end{proof}
We now turn to identifying pairs and quartets of 
permutations
that make no net contribution to the \mob sum.
We start by showing that 
if $\sigma$ and $\alpha$ are indecomposable,
and $r \geq 1$, and
with $\pi \in \familysum{\sumra}$,
then
$\mobfn{\sigma}{\pi}$ and $\mobfn{\sigma}{\alpha}$
have the same magnitude.
\begin{lemma}
    \label{lemma_pi_has_single_block}
    Let $\pi \in \familysum{\sumra}$,
    where $r \geq 1$ and $\alpha > 1$ is sum indecomposable.
    Let $\sigma$ be a sum indecomposable permutation. 
    Then 
    \[
    \mobfn{\sigma}{\pi} = 
    \begin{cases*}
    \mobfn{\sigma}{\alpha} & 
    \text{if $\pi = \sumra$\; or\; $\oneplus \sumrab \plusone $},
    \\
    -\mobfn{\sigma}{\alpha} &
    \text{if $\pi = \oneplus \sumrab $\; or\; $\sumrab \plusone$}.
    \end{cases*}
    \]
    
    As a consequence, $\familysum{\sumra}$ makes no net contribution to 
    $\mobfn{\sigma}{\pi}$ if
    $\familysum{\sumra} \subseteq [\sigma,\pi)$.
\end{lemma}
\begin{proof}
    If $\pi = \sumra$
    or $\pi = \sumrab \plusone$,
    then this is immediate from Corollary~\ref{BJJS_corollary_3}.
    If $\pi = \oneplus \sumrab$
    or $\pi = \oneplus \sumrab \plusone$,    
    then we use Proposition~\ref{BJJS_proposition_1}.
    
    For the net contribution of $\familysum{\sumra}$,
    $\sum_{\lambda \in \familysum{\sumra}} \mobfn{\sigma}{\lambda} = 0$.
\end{proof}
We now have a lemma that adds a further restriction to
the permutations that have a non-zero
contribution to the \mob sum.
\begin{lemma}
    \label{lemma_only_r_and_rplusone_count}
    If $\sigma \leq \pi$, 
    and $\alpha \in [\sigma, \pi]$ is sum indecomposable,
    and $r$ is the smallest integer such that 
    $\oneplus \sumrab \plusone \not\leq \pi$, then
    $\familysum{\nsums{k}{\alpha}} \subseteq [\sigma, \pi)$
    for all $k \in [1, r)$.
\end{lemma}
\begin{proof}
        For any $k < r$, 
        $\sigma \leq 
        \nsums{k}{\alpha} < 
        \oneplus \left(\nsums{k}{\alpha}\right) \plusone \leq
        \pi$.    
        Note that by Lemma~\ref{lemma_pi_has_single_block}
        the net contribution of the family 
        $\familysum{\nsums{k}{\alpha}}$ to $\mobfn{\sigma}{\pi}$ is zero.
\end{proof}
\begin{observation}
    \label{observation_only_r_and_rplusone_count}
    Using the same terminology as 
    Lemma~\ref{lemma_only_r_and_rplusone_count}, if $k > r+1$
    then we must have $\nsums{k}{\alpha} \not\leq \pi$.
    As a consequence, for each indecomposable 
    $\alpha \in [\sigma, \pi]$, the only families of $\alpha$
    that can have a non-zero net contribution
    to $\mobfn{\sigma}{\pi}$ are
    $\familysum{\sumra}$ 
    and $\familysum{\nsums{r+1}{\alpha}}$.
\end{observation}
We now eliminate two specific permutations from the \mob sum.
\begin{lemma}
    \label{lemma_order_greater_than_three_ignore_o_opo}
    If $\pi$ is any permutation with $\order{\pi} > 3$
    apart from the identity permutation and its reverse,
    and $\sigma$ is sum indecomposable,
    then the permutations $1$ and $\oneplusone$ make no net contribution
    to the \mob sum $\mobfn{\sigma}{\pi}$.
\end{lemma}
\begin{proof}
    If $\sigma = 1$, then the interval contains
    both $1$ and $\oneplusone$.  Since 
    $\mobfn{1}{1} = 1$ and $\mobfn{1}{12} = -1$,
    there is no net contribution to $\mobfn{\sigma}{\pi}$.
    If $\sigma > 1$, then $\sigma \neq 12$,
    and so neither $1$ not $12$ is in the interval.
\end{proof}

Before we present the main theorem for this section,
we formally define the weight function
and the contributing set.
Let $\alpha$ be a sum indecomposable permutation.
The weight function, $\weightgen{\sigma}{\alpha}{\pi}$,
is defined as 
\begin{align}
\label{equation_general_weight_function}
\weightgen{\sigma}{\alpha}{\pi}
& =
\begin{cases*}
1 & 
If
$
\left\lbrace
\begin{array}{l}
\sigma \leq \sumra \leq \pi
\text{ and } \\
\oneplus \sumrab \not \leq \pi
\text{ and } \\
\sumrab \plusone \not \leq \pi,
\end{array}
\right.
$
\\
-1 & 
If
$
\left\lbrace
\begin{array}{l}
\sigma \leq \sumra \leq \pi
\text{ and } \\
\oneplus \sumrab \leq \pi
\text{ and } \\
\sumrab \plusone \leq \pi 
\text{ and } \\
\nsums{r+1}{\alpha} \not\leq \pi,
\end{array}
\right.
$
\\
0 &
Otherwise,
\end{cases*}
\end{align}
where $r$ is the smallest integer 
such that $\oneplus \sumrab \plusone \not \leq \pi$.

The contributing set $\contrib{\sigma}{\pi}$ 
is defined as
\begin{align*}
\contrib{\sigma}{\pi}
& = 
\left\lbrace
\alpha :
\begin{array}{l}
\alpha \in [\sigma, \pi), \\
\alpha \text{ is sum indecomposable, and } \\
\weightgen{\sigma}{\alpha}{\pi} \neq 0
\end{array}
\right\rbrace.
\end{align*}

We have one last lemma before we move on
to the main theorem.
\begin{lemma}
    \label{lemma_weight_of_alpha}
    If $\sigma$ and $\alpha$ are sum indecomposable,
    then for any permutation $\pi$,
    $\mobfn{\sigma}{\alpha}\weightgen{\sigma}{\alpha}{\pi}$
    gives the contribution of the set of 
    families
    $\familysum{\sumra}$ to the \mob sum,
    where $r$ is any positive integer.
\end{lemma}
\begin{proof}
    By Observation~\ref{observation_only_r_and_rplusone_count},
    we only need consider the contribution made
    by $\sumra$ and $\nsums{r+1}{\alpha}$,
    where $r$ is the smallest integer 
    such that $\oneplus \sumrab \plusone \not \leq \pi$.
    
    If $\sigma \not \leq \sumra$, 
    or $\sumra \not \leq \pi$, then
    $\familysum{\sumra}$ makes no net contribution 
    to the \mob sum.  Now assume that
    $\sigma \leq \sumra \leq \pi$.
    First, we can see that
    if $\oneplus \sumrab \not \leq \pi$,
    or $\sumrab \plusone \not \leq \pi$
    then $\nsums{r+1}{\alpha} \not \leq \pi$.
    We can also see that 
    if $\oneplus \sumrab \plusone \not \leq \pi$
    then $\oneplus \left(\nsums{r+1}{\alpha}\right) \not \leq \pi$
    and $\nsums{r+1}{\alpha} \not \leq \pi$.
    The possibilities remaining are itemised
    in Table~\ref{table_family_members_in_an_interval},
    \begin{table}[btp]
        \centering
        \begin{tabular}{cccc}
            \toprule
            $\oneplus \sumrab$ &
            $\sumrab \plusone$ &
            $\nsums{r+1}{\alpha}$ &
            \mob contribution \\    
            \midrule
            $\leq \pi$ &
            $\leq \pi$ &
            $\leq \pi$ &
            $0$
            \\    
            $\leq \pi$ &
            $\leq \pi$ &
            $\not \leq \pi$ &
            $- \mobfn{\sigma}{\alpha}$
            \\    
            $\leq \pi$ &
            $\not \leq \pi$ &
            $\not \leq \pi$ &
            $0$
            \\    
            $\not \leq \pi$ &
            $\leq \pi$ &
            $\not \leq \pi$ &
            $0$
            \\    
            $\not \leq \pi$ &
            $\not \leq \pi$ &
            $\not \leq \pi$ &
            $\mobfn{\sigma}{\alpha}$
            \\    
            \bottomrule
        \end{tabular}    
        \caption{\mob contribution from family members.}
        \label{table_family_members_in_an_interval}
    \end{table}      
    where the \mob contribution is determined
    by applying 
    Lemma~\ref{lemma_pi_has_single_block}.    
    We can see that in every case 
    $\weightgen{\sigma}{\alpha}{\pi}$ provides the
    correct weight for the \mob function
    $\mobfn{\sigma}{\alpha}$.
\end{proof}

We are now in a position 
to present the main theorem for this section.
\begin{theorem}
\label{theorem_mobius_sum_bottom_level_indecomposable}
    If $\sigma$ is a sum indecomposable permutation,
    and $\order{\pi} > 3$,
    then
    \[
    \mobfn{\sigma}{\pi} 
    = 
    - \sum_{\alpha \in \contrib{\sigma}{\pi}}
    \mobfn{\sigma}{\alpha} \weightgen{\sigma}{\alpha}{\pi} .
    \] 
\end{theorem}
\begin{proof}
    Let $\alpha \leq \pi$ be an indecomposable permutation.
    
    Using Lemmas~\ref{lemma_mobius_function_is_zero}
    and~\ref{lemma_order_greater_than_three_ignore_o_opo}
    we can see that any permutations not in the set
    $\familysum{\sumra}$
    can be excluded from $\contrib{\sigma}{\pi}$, as these permutations 
    make no net contribution to the \mob sum.
    
    For every $\alpha$, 
    by Lemma~\ref{lemma_weight_of_alpha},
    $\mobfn{\sigma}{\alpha} \weightgen{\sigma}{\alpha}{\pi}$
    provides the contribution to the \mob sum
    of all families $\familysum{\sumra}$, where
    $r$ is a positive integer.  
\end{proof}
Theorem~\ref{theorem_mobius_sum_bottom_level_indecomposable}
reduces the number of permutations that need to be considered
as part of the \mob sum.  
We can see that the largest permutation in 
$\contrib{\sigma}{\pi}$ must have length less than $\order{\pi}$,
and so we can apply 
Theorem~\ref{theorem_mobius_sum_bottom_level_indecomposable}
recursively to the permutations in $\contrib{\sigma}{\pi}$
to determine their \mob values.
In this recursion, if we are attempting to determine
$\mobfn{\sigma}{\lambda}$, we can stop
if $\order{\sigma} = \order{\lambda}$
or $\order{\sigma} = \order{\lambda} - 1$, as
in these cases $\mobfn{\sigma}{\lambda}$ is $+1$ and $-1$ respectively.

\section{Increasing oscillations}\label{section-increasing-oscillations}

We now move on to increasing oscillations.
Given an indecomposable permutation
$\sigma$, 
and an increasing oscillation $\pi$,
our aim in this section is to describe
$\contrib{\sigma}{\pi}$
in precise terms.  
We will find a sum for the \mob function,
$\mobfn{\sigma}{\pi}$,
which only requires the evaluation of simple inequalities.

If $\pi$ is an increasing oscillation with 
length less than 4, then
$\mobfn{\sigma}{\pi}$ is trivial to determine
for any $\sigma$.
For the remainder of this section we assume that 
$\pi$ has length at least 4.

We partition the set of increasing oscillations
with length greater than 1 
into 
five disjoint subsets.
These subsets are
$\{ \dtwo \},\;$
$\{\ildtwokp \},\;$ 
$\{\oneil \ildtwokb\},\;$ 
$\{\ildtwokb \ilone\},\;$
and
$\{\oneil \ildtwokb \ilone \}$,
where $k$ is a positive integer.
If two increasing oscillations are in the 
same subset, then we
say that they have the same \emph{shape}.

We now determine what 
permutations contained in an increasing oscillation
have a non-zero contribution to the \mob sum.
\begin{lemma}
    \label{lemma_form_of_permutations_in_increasing_oscillations}
    Let $\pi$ be an increasing oscillation,
    and let $\sigma \leq \pi$ be sum indecomposable.
    Let $S$ be the subset of the permutations
    in the interval $[\sigma, \pi)$ that can be 
    written in the form
    $
    \familysum{ \nsums{r}{\familyil{\nils{k}{\dtwo}} } }
    $
    for some $k,r \geq 1$.
    If $\lambda \in [\sigma, \pi)$,
    and $\lambda \not \in S$, then
    $\mobfn{\sigma}{\lambda} = 0$.
\end{lemma}
We note here that 
$\familyil{\nils{k}{\dtwo}}$
is a set containing only increasing oscillations. 

\begin{proof}
    We start by showing that if $\pi$ is an increasing oscillation,
    and
    $\lambda = \lambda_1 \oplus \ldots \oplus \lambda_m \leq \pi$,
    where each $\lambda_i$ is sum indecomposable,
    then every $\lambda_i$
    is an increasing oscillation.
    This is trivially true if $\lambda$
    is itself an increasing oscillation,
    thus it is 
    sufficient to show
    that if  
    $\lambda$ is an increasing oscillation, 
    then
    deleting a single point results in either 
    an increasing oscillation, or
    a permutation that is the sum of
    two increasing oscillations.
    
    If $k = 1$, 
    then we can see that deleting a single point
    results in a permutation with the required 
    characteristic.
    
    Now assume that $k > 1$.
    Let $\lambda = \oneil \ildtwokb$.
    Deleting the leftmost point gives $\ildtwok$,
    and deleting the rightmost point gives
    $\oneil \left( \nils{k-1}{\dtwo} \right) \oplus 1$.
    Deleting the second point
    gives
    $\dtwo \oplus \left( \nils{k-1}{\dtwo} \right)$,
    and deleting the last-but-one point
    gives
    $\oneil \left( \nils{k-1}{\dtwo} \right) \ilone$.
    Deleting any even point $2t$ 
    except the second or second-to-last
    results in 
    $
    \left(
    \oneil \left( \nils{t-1}{\dtwo} \right) \ilone 
    \right)
    \oplus
    \left(
    \left( \nils{k-t}{\dtwo} \right)
    \right)
    $.
    Finally, deleting any odd point $2t+1$
    apart from the first or last
    results in
    $
    \left(
    \oneil \left( \nils{t-1}{\dtwo} \right)
    \right)
    \oplus
    \left(
    \oneil \left( \nils{k-t}{\dtwo} \right)
    \right)
    $.
    Thus if $\lambda = \oneil \ildtwokb$,
    then deleting a single point from $\lambda$
    results in either 
    an increasing oscillation, or
    a permutation that is the sum of
    two increasing oscillations.
    
    A similar argument applies to the other three cases,
    which we omit for brevity.
       
    To complete the proof, 
    we now see that by
    Lemma~\ref{lemma_order_greater_than_three_ignore_o_opo},
    we can ignore $\lambda = 1 $ and $\lambda = \oneplusone$.    
    If 
    $\lambda = \lambda_1 \oplus \lambda_2 \oplus \ldots \oplus \lambda_m \leq \pi$,
    then by the argument above, every $\lambda_i$
    is an increasing oscillation.
    Applying Lemma~\ref{lemma_mobius_function_is_zero} 
    completes the proof.
\end{proof}
Following Observation~\ref{observation_only_r_and_rplusone_count},
it is clear that,
if $\alpha \in \familyil{\ildtwok}$,
then for any family $\familysum{\sumra}$,
we only need consider the cases 
$\sumra$ and $\nsums{r+1}{\alpha}$
where
$r$ is the smallest integer such that
$\oneplus \sumrab \plusone \not \leq \pi$.

Given some $\pi = W_n \text{ or } M_n$, we will 
find inequalities that relate $n$, $r$ and $k$ and the
shape of $\alpha$ that will allow us to
find the values that contribute
to the \mob sum.
We know from
Lemma~\ref{lemma_form_of_permutations_in_increasing_oscillations}
the shape of the permutations that contribute to the \mob sum.
For each of the four types of increasing oscillation
($W_{2n}$, $W_{2n-1}$, $M_{2n}$ and $M_{2n-1}$),
we can examine how each shape can be embedded 
so that the unused points at the start of the increasing
oscillation are minimised.  
Figure~\ref{figure_unused_points_at_start_w2n}
shows examples of embeddings into $W_{2n}$.
This gives us an inequality relating to the start
of the embedding.
Similarly, we can find inequalities for the 
end of the embedding.
We can also find inequalities 
that relate to the interior (when $r > 1$), and
Figures~\ref{figure_unused_points_interior_not_21}
and~\ref{figure_unused_points_interior_21}
show examples of this.
We can use these inequalities to determine what
values of $k$ will allow the shape to be embedded.
For each allowable value of $k$, we
can then determine the maximum value of $r$
such that $\oneplus \sumrab \plusone \not \leq \pi$.
This then means that, by evaluating inequalities alone,
we can identify the specific permutations
that could contribute to the \mob sum.

We first have two lemmas that examine
inequalities at the start and end of an embedding.
  
\begin{lemma}
    \label{lemma_unused_points_at_start}
    If $\pi$ is an increasing oscillation,
    and $\alpha \leq \pi$ is sum indecomposable, 
    then in any embedding of an 
    element $\lambda$ of $\familysum{\sumra}$ into $\pi$,
    the minimum number of unused points at the start of $\pi$
    depends on the start of $\lambda$, and on $\pi$,
    and is as shown below:
    
    \centering
    \begin{tabular}{lcccc}
        \toprule
        Start of $\lambda$ & $\pi=W_{2n}$ & $\pi=W_{2n-1}$ & $\pi=M_{2n}$ & $\pi=M_{2n-1}$ \\
        \midrule
        $\dtwo \ldots $ & $0$ & $0$ & $0$ & $0$ \\
        $\ildtwokp \ldots $ & $0$ & $0$ & $1$ & $1$ \\
        $\oneil \left( \ildtwok \right) \ldots $ & $1$ & $1$ & $0$ & $0$ \\
        \midrule
        $\oneplus \dtwo \ldots $ & $1$ & $1$ & $1$ & $1$ \\
        $\oneplus \left( \ildtwokp \right) \ldots$ & $1$ & $1$ & $2$ & $2$ \\
        $\oneplus \oneil \left( \ildtwok \right) \ldots$ & $2$ & $2$ & $1$ & $1$ \\
        \bottomrule
    \end{tabular}        
\end{lemma}
\begin{proof}
    It is clear that if we
    minimise the number of points at the start of
    an embedding, then the number of unused points
    depends on $\pi$, 
    and the start of $\alpha$.  
    The values in Lemma~\ref{lemma_unused_points_at_start}
    are found by considering each of the possibilities.
    We illustrate some of these cases in Figure~\ref{figure_unused_points_at_start_w2n}.
    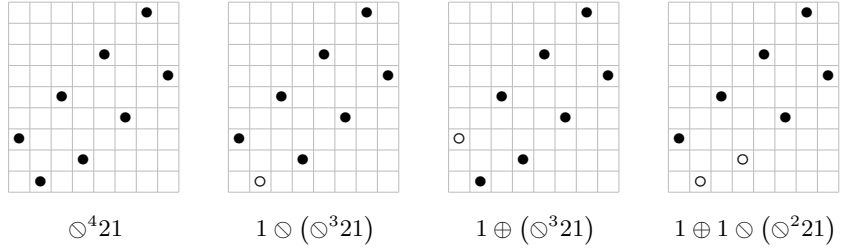
\begin{figure}      
        \centering
        \begin{subfigure}[c]{0.2\textwidth}
            \begin{tikzpicture}[scale=0.28]
            \plotgrid{8}{9}
            \plotperm{3, 1, 5, 2, 7, 4, 9, 6}
            \end{tikzpicture}
            \caption*{$\nils{4}{\dtwo}$}
        \end{subfigure}
        \quad
        \begin{subfigure}[c]{0.2\textwidth}
            \begin{tikzpicture}[scale=0.28]
            \plotgrid{8}{9}
            \plotperm{3, 1, 5, 2, 7, 4, 9, 6}
            \opendot{(2,1)}
            \end{tikzpicture}
            \caption*{$\oneil \left( \nils{3}{\dtwo} \right)$}
        \end{subfigure}
        \quad
        \begin{subfigure}[c]{0.2\textwidth}
            \begin{tikzpicture}[scale=0.28]
            \plotgrid{8}{9}
            \plotperm{3, 1, 5, 2, 7, 4, 9, 6}
            \opendot{(1,3)}
            \end{tikzpicture}
            \caption*{$\oneplus \left( \nils{3}{\dtwo} \right)$}
        \end{subfigure}
        \quad
        \begin{subfigure}[c]{0.2\textwidth}
            \begin{tikzpicture}[scale=0.28]
            \plotgrid{8}{9}
            \plotperm{3, 1, 5, 2, 7, 4, 9, 6}
            \opendot{(2,1)}
            \opendot{(4,2)}
            \end{tikzpicture}
            \caption*{$\oneplus \oneil \left( \nils{2}{\dtwo} \right)$}
        \end{subfigure}    
        \caption{Embedding the start of $\alpha$ in $W_{2n}$.}
        \label{figure_unused_points_at_start_w2n}
    \end{figure} 
\end{proof}
%
\begin{lemma}
    \label{lemma_unused_points_at_end}
    If $\pi$ is an increasing oscillation,
    and $\alpha \leq \pi$ is sum indecomposable, 
    then in any embedding of an 
    element $\lambda$ of $\familysum{\sumra}$ into $\pi$,
    the minimum number of unused points at the end of $\pi$
    depends on the end of $\lambda$, and on $\pi$,
    and is as shown below:    
    
    \centering
    \begin{tabular}{lcccc}
        \toprule
        End of $\lambda$                     & 
        $\pi=W_{2n}$ & $\pi=W_{2n-1}$ & $\pi=M_{2n}$ & $\pi=M_{2n-1}$ \\
        \midrule
        $\ldots \dtwo $ & $0$ & $0$ & $0$ & $0$ \\
        $\ldots \ildtwokp $ & $0$ & $1$ & $1$ & $0$ \\
        $\ldots \ildtwokb \ilone $ & $1$ & $0$ & $0$ & $1$ \\
        \midrule
        $\ldots \dtwo \plusone$ & $1$ & $1$ & $1$ & $1$ \\
        $\ldots \ildtwokpb \plusone $ & $1$ & $2$ & $2$ & $1$ \\
        $\ldots \ildtwokb \ilone \plusone$ & $2$ & $1$ & $1$ & $2$ \\
        \bottomrule
    \end{tabular}
\end{lemma}
\begin{proof}
    We examine all the possibilities as we did in
    Lemma~\ref{lemma_unused_points_at_start}.
\end{proof}

We now consider how closely copies of 
some sum indecomposable $\alpha$ can be 
embedded into $\pi$.
This leads to two inequalities
that relate $\alpha$, $\pi$
and the maximum number of copies of $\alpha$
that can be embedded in $\pi$.
Where $\alpha \neq \dtwo$, the shape
of $\alpha$ fixes the way the two copies
can be embedded in an increasing oscillation.
If $\alpha = \dtwo$, then we will see that
there are choices for the embedding.
\begin{lemma}
    \label{lemma_unused_points_interior}
    If $\pi$ is an increasing oscillation,
    and $\alpha \neq \dtwo$,
    and $\alpha \leq \pi$ is sum indecomposable, 
    then in any embedding of $\sumra$ into $\pi$,
    the minimum number of points 
    between the start and end of $\sumra$
    depends on  $\alpha$,
    and is as shown below:
    
    \centering
    \begin{tabular}{lccc}
        \toprule
        Shape of $\alpha$ & 
        Points in $\sumra$ &
        Unused points &
        Minimum points
        \\ 
        \midrule
        $\ildtwokp$ & $2kr $ & $2r-2$ & $2kr+2r-2$ \\
        $\oneil \ildtwokb$ & $2kr+r $ & $r-1 $ & $2kr+2r-1$ \\
        $\ildtwokb \ilone$ & $2kr+r $ & $r-1 $ & $2kr+2r-1$ \\
        $\oneil \ildtwokb \ilone$ & $2kr+2r $ & $2r-2$ & $2kr+4r-2$ \\
        \bottomrule
    \end{tabular}    
\end{lemma}
\begin{proof}
    If $r$ = 1, then there are no unused points,
    and so the minimum number of points depends solely
    on the points in $\alpha$, and the table 
    reflects this.
    
    Assume now that $r > 1$.
    If $\alpha \neq \dtwo$, then we can see that the 
    interleave fixes the layout of each copy of $\alpha$,
    so we simply pack each copy as close as possible.
    This packing clearly depends on the 
    start and end of $\alpha$, and it
    is simple to examine the four possibilities.
    Examples are shown in Figure~\ref{figure_unused_points_interior_not_21}.
    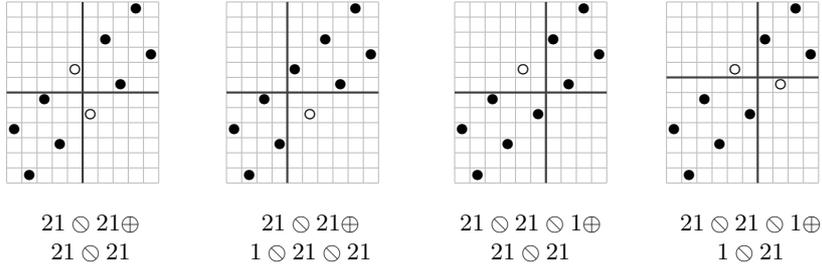
\begin{figure}
        \centering
        \captionsetup[subfigure]{justification=centering}
        \begin{subfigure}[c]{0.2\textwidth}        
            \begin{tikzpicture}[scale=0.2]
            \plotgrid{10}{12}
            \plotperm{4,1,6,3,8,5,10,7,12,9}
            \opendot{(5,8)}
            \opendot{(6,5)}
            \darkhline{6}{10}
            \darkvline{5}{12}
            \end{tikzpicture}
            \caption*{$\dtwo \interleave \dtwo \oplus$\\$\dtwo \interleave \dtwo$}
        \end{subfigure}
        \quad
        \begin{subfigure}[c]{0.2\textwidth}
            \begin{tikzpicture}[scale=0.2]
            \plotgrid{10}{12}
            \plotperm{4,1,6,3,8,5,10,7,12,9}
            \opendot{(6,5)}
            \darkhline{6}{10}
            \darkvline{4}{12}
            \end{tikzpicture}
            \caption*{$\dtwo \interleave \dtwo \oplus$\\$\oneil \dtwo \interleave \dtwo$}
        \end{subfigure}
        \quad
        \begin{subfigure}[c]{0.2\textwidth}
            \centering
            \begin{tikzpicture}[scale=0.2]
            \plotgrid{10}{12}
            \plotperm{4,1,6,3,8,5,10,7,12,9}
            \opendot{(5,8)}
            \darkhline{6}{10}
            \darkvline{6}{12}
            \end{tikzpicture}
            \caption*{$\dtwo \interleave \dtwo \ilone \oplus$\\$\dtwo \interleave \dtwo$}
        \end{subfigure}
        \quad
        \begin{subfigure}[c]{0.2\textwidth}
            \begin{tikzpicture}[scale=0.2]
            \plotgrid{10}{12}
            \plotperm{4,1,6,3,8,5,10,7,12,9}
            \opendot{(5,8)}
            \opendot{(8,7)}
            \darkhline{7}{10}
            \darkvline{6}{12}
            \end{tikzpicture}
            \caption*{$\dtwo \interleave \dtwo \ilone \oplus$\\$\oneil \dtwo$}
        \end{subfigure}
        \caption{Packing $\alpha$ as close as possible when $\alpha \neq \dtwo$.}
        \label{figure_unused_points_interior_not_21}
    \end{figure} 
\end{proof}

We now turn to the case where $\alpha = \dtwo$.  
This is more complex than the previous cases.
We can see that
there must be at least one point between each copy of $\alpha$.
We can insert each copy of $\dtwo$
in two ways,
one where the points are horizontally adjacent,
and one where the points are vertically adjacent.
These alternatives can be seen in 
Figure~\ref{figure_unused_points_interior_21}.
Alternating these means that there will be exactly
one point between each copy of $\alpha$,
so this embedding minimises the number of 
points between the start and end of $\sumra$.
The complication in this case relates to how we
start and end the embedding.  We illustrate this by
showing, in Figure~\ref{figure_unused_points_interior_21},
maximal embeddings where
we are embedding into $W_{8}$, $W_{10}$ and $W_{12}$.
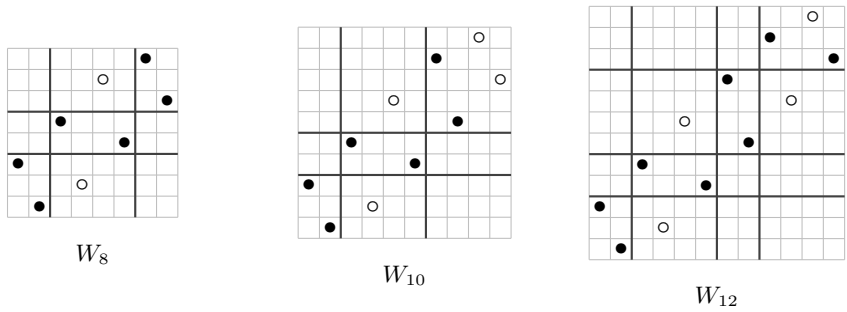
\begin{figure}       
    \begin{center}
        \begin{subfigure}[c]{0.3\textwidth}
            \begin{center}
            \begin{tikzpicture}[scale=0.28]
            \plotgrid{8}{8}
            \plotperm{3, 1, 5, 2, 7, 4, 8, 6}
            \opendot{(4,2)}
            \opendot{(5,7)}
            \darkhline{3}{8}
            \darkhline{5}{8}
            \darkvline{2}{8}
            \darkvline{6}{8}
            \end{tikzpicture}
            \caption*{$W_{8}$}
            \end{center}
        \end{subfigure}
        \quad
        \begin{subfigure}[c]{0.3\textwidth}
            \begin{center}
            \begin{tikzpicture}[scale=0.28]
            \plotgrid{10}{10}
            \plotperm{3, 1, 5, 2, 7, 4, 9, 6, 10, 8}
            \opendot{(4,2)}
            \opendot{(5,7)}
            \opendot{(9,10)}
            \opendot{(10,8)}
            \darkhline{3}{10}
            \darkhline{5}{10}
            \darkvline{2}{10}
            \darkvline{6}{10}
            \end{tikzpicture}
            \caption*{$W_{10}$}
            \end{center}
        \end{subfigure}
        \quad
        \begin{subfigure}[c]{0.3\textwidth}
            \begin{center}
            \begin{tikzpicture}[scale=0.28]
            \plotgrid{12}{12}
            \plotperm{3, 1, 5, 2, 7, 4, 9, 6, 11, 8, 12, 10}
            \opendot{(4,2)}
            \opendot{(5,7)}
            \opendot{(10,8)}
            \opendot{(11,12)}
            \darkhline{3}{12}
            \darkhline{5}{12}
            \darkhline{9}{12}
            \darkvline{2}{12}
            \darkvline{6}{12}
            \darkvline{8}{12}
            \end{tikzpicture}
            \caption*{$W_{12}$}
            \end{center}
        \end{subfigure}
        \caption{Examples of unused points when embedding $\nsums{r}{\dtwo}$.}
        \label{figure_unused_points_interior_21}
    \end{center}
\end{figure}      
A detailed examination of each possible case
gives us our second inequality.
\begin{lemma}
    \label{lemma_unused_points_21}
    If $\pi$ is an increasing oscillation,
    and $\alpha = \dtwo$
    then for $\sumra$ to be contained
    in $\pi$
    we must have 
    $3r - 1 \leq 2n$ for $\pi \in \{W_{2n}, M_{2n} \}$,
    and
    $3r \leq 2n$ for $\pi \in \{W_{2n-1}, M_{2n-1} \}$.
\end{lemma}
\begin{proof}
    In every case we start by embedding the first $\dtwo$
    into the first two elements of the permutation.
    Thereafter, we embed each successive $\dtwo$
    as close as possible to the preceding $\dtwo$.
    The minimum number of elements to embed $r$ copies
    of $\dtwo$ will be $2r$ elements to hold the 
    points of the $\dtwo$s, and $r-1$ intermediate empty elements.
    For $W_{2n}$ and $M_{2n}$, this then gives
    $3r-1 \leq 2n$, and 
    for $W_{2n-1}$ and $M_{2n-1}$ we obtain
    $3r-1 \leq 2n-1$.
\end{proof}

We now have a complete understanding of the number of points
required to embed any 
permutation that contributes to the \mob sum
into an increasing oscillation.
The following Lemma summarises the situation.

\begin{lemma}
    \label{lemma_inequalities_for_pi}
    If $\pi$ is an increasing oscillation,
    and $\alpha \in \familyil{\ildtwok} \leq \pi$
    (so $\alpha$ is sum indecomposable),
    then for $\sumra$ to be contained
    in $\pi$,
    the inequality in the table below must be satisfied,
    where $k \geq 1$.
    
    \centering
    \begin{tabular}{l@{\phantom{xxxxxxx}}c@{\phantom{xxxxxxx}}r}
        \toprule
        $\pi$ & Shape of $\alpha$ & Inequality \\ 
        \midrule
        $W_{2n},   M_{2n}$    & $\dtwo$ &  $3r -1 \leq 2n $ \\
        $W_{2n-1}, M_{2n-1}$  & $\dtwo$ &  $3r    \leq 2n $\\
        $W_{2n}$ & $ \ildtwokp $ & $ 2kr+2r-2 \leq 2n $ \\
        $W_{2n-1}$ & $ \oneil \ildtwokb $ & $ 2kr+2r+2 \leq 2n $ \\
        $M_{2n-1}$ & $ \ildtwokb \ilone $ & $ 2kr+2r+2 \leq 2n $ \\
        $M_{2n}$ & $ \oneil \ildtwokb \ilone $ & $ 2kr+4r-2 \leq 2n $ \\ 
        $W_{2n}, W_{2n-1}, M_{2n-1}$ & $ \oneil \ildtwokb \ilone $ & $ 2kr+4r \leq 2n $ \\
        \multicolumn{2}{c}{All other cases} & $ 2kr+2r \leq 2n $ \\
        \bottomrule
    \end{tabular}
\end{lemma}
\begin{proof}
    We apply Lemmas~\ref{lemma_unused_points_at_start},
    \ref{lemma_unused_points_at_end}, 
    ~\ref{lemma_unused_points_interior} 
    and~\ref{lemma_unused_points_21} to the
    possibilities for $\pi$ and $\alpha$.
\end{proof}

As a consequence of 
Lemmas~\ref{lemma_unused_points_at_start}
and~\ref{lemma_unused_points_at_end}
we can define a relationship
between the minimum number of points
required to embed some $\sumra$,
and the minimum number of points required
to embed 
$\oneplus \sumrab$,
$\sumra \plusone$
and
$\oneplus \sumrab \plusone$.
\begin{corollary}
    \label{corollary_add_2_for_op_or_po}
    If $\pi$ is an increasing oscillation,
    and $\alpha \leq \pi$ is sum indecomposable
    and if
    the minimum number of points required to embed $\sumra$ into $\pi$
    is $C$, then
    the minimum number of points required to embed 
    $\oneplus \sumrab$ into $\pi$ is $C+2$,
    the minimum number of points required to embed 
    $\sumra \plusone$ into $\pi$ is $C+2$,
    and
    the minimum number of points required to embed 
    $\oneplus \sumrab \plusone$ into $\pi$ is $C+4$.
\end{corollary}
\begin{proof}
    We can see from Lemmas~\ref{lemma_unused_points_at_start}
    and~\ref{lemma_unused_points_at_end} that 
    adding $\oneplus {}$ 
    at the start of a permutation
    increases the number of points required by two -- 
    one for the new point, and one that is unused.
    Similarly, adding ${} \plusone$
    at the end increases the points required by two.
\end{proof}

Lemma~\ref{lemma_inequalities_for_pi} gives
us inequalities that any $\sumra$ must satisfy
to ensure that $\sumra \leq \pi$.  Further,
Corollary~\ref{corollary_add_2_for_op_or_po}
gives us inequalities that, for a given $\sumra$
allow us to determine if 
$\oneplus \sumrab \leq \pi$, 
$\sumrab \plusone \leq \pi$ and
$\oneplus \sumrab \plusone \leq \pi$.
We can therefore determine what values
of $r$ and $k$ will result in 
$\familysum{\sumra}$ contributing to the \mob function.
We now consider inequalities that
relate $\sigma$ and $\alpha$,
so that we can determine
if $\sigma \leq \alpha$ 
using an inequality.
\begin{lemma}
    \label{lemma_inequalities_for_sigma}
    If $\sigma > 1$ is an increasing oscillation,
    and $\alpha \in \familyil{\ildtwok}$ for some $k$,
    then for $\sigma$ to be contained
    in $\alpha$
    the inequality in the table below must be satisfied,
    where $k \geq 1$.
    
    \centering
    \begin{tabular}{l@{\phantom{xxxxxxx}}c@{\phantom{xxxxxxx}}r}
        \toprule
        $\sigma$ & Shape of $\alpha$ & Inequality \\ 
        \midrule
        $W_{2n-1},M_{2n},M_{2n-1}$ & $ \dtwo $ & False \\        
        $W_{2n-1}$ & $ \ildtwokb \ilone $ & $ k \geq n-1 $ \\
        $M_{2n-1}$ & $ \oneil \ildtwokb $ & $ k \geq n-1 $ \\
        $W_{2n-1}, M_{2n}, M_{2n-1}$ & $ \oneil \ildtwokb \ilone $ & $ k \geq n-1 $ \\         
        $M_{2n}$   & $ \ildtwokp $ & $ k \geq n+1 $ \\
        \multicolumn{2}{c}{All other cases} & $ k \geq n $ \\
        \bottomrule
    \end{tabular}
\end{lemma}
\begin{proof}
    We examine all possible cases.
\end{proof}

We are now nearly ready to present the main theorem 
for this section.  
Informally, for each possible shape
of permutation $\alpha$,
we will first find the
minimum and maximum values of $k$
such that $\sigma \leq \alpha \leq \pi$,
as any other values of $k$ result in 
$\alpha$ being outside the interval.
For each $\alpha$ and each $k$, we then 
determine the minimum value of $r$ such that
$\oneplus \sumrab \plusone \not\leq \pi$.
We can then use this value of $r$ (assuming it is non-zero)
to determine the weight to be applied to 
$\mobfn{\sigma}{\alpha}$.
The set of $\alpha$s with a non-zero weight
is a contributing set $\contrib{\sigma}{\pi}$.
At this point we can substitute a value
for any $\mobfn{\sigma}{\alpha}$
where $\order{\sigma} \leq \order{\alpha} - 1$.
We then use the same process
recursively to determine
the contributing set for the
remaining elements of $\contrib{\sigma}{\pi}$.

We first define some supporting functions.
Let $\rawmink(\sigma, \alpha)$
be the 
minimum value of $k$ that satisfies the inequality
in Lemma~\ref{lemma_inequalities_for_sigma}.
For the first inequality, which is always false,
we set $k=\order{\pi}$, as this will force
the sum, defined later in
Theorem~\ref{theorem_mobius_sum_increasing_oscillations},
to be empty.

Let $\mink(\sigma, \alpha)$ be defined as
\begin{align*}
\mink(\sigma, \alpha)
& =
\begin{cases*}
1 & If $\sigma = 1$ and $\alpha \neq \ildtwokp$, \\
2 & If $\sigma = 1$ and $\alpha = \ildtwokp$, \\
\rawmink(\sigma,\alpha) & otherwise.
\end{cases*}
\end{align*}
Observe that for any $k < \mink(\sigma, \alpha)$,
we have $\alpha < \sigma$, and so 
$\familysum{\nsums{k}{\alpha}}$ makes no net contribution to the \mob sum.

Let $\maxk(\alpha, \pi)$ be defined as
the maximum value of $k$ that satisfies the inequality in
Lemma~\ref{lemma_inequalities_for_pi},
if the shape of $\alpha$ and the shape of $\pi$ are different;
and
one less than the maximum value of $k$ that satisfies the inequality 
if the shape of $\alpha$ and the shape of $\pi$ are the same.    
For the first two inequalities, which do not involve $k$,
we set $\maxk(\alpha, \pi)=1$ if the inequality is satisfied, 
and $\maxk(\alpha, \pi)=0$ if not.
Observe here that 
for any $k > \maxk(\alpha, \pi)$ we have
$\alpha \not < \pi$, and so
$\familysum{\nsums{k}{\alpha}}$ makes no contribution to the \mob sum.

We define the weight function
for increasing oscillations, $\weightosc{\sigma}{\alpha}{\pi}$,
as
\begin{align*}
\weightosc{\sigma}{\alpha}{\pi}
& =
\begin{cases*}
1 & 
If
$
\sumrab \plusone \not \leq \pi,
$
\\
-1 & 
If
$
\sumrab \plusone \leq \pi 
\text{ and } 
\nsums{r+1}{\alpha} \not\leq \pi,
$
\\
0 &
Otherwise,
\end{cases*}
\end{align*}
where $r$ is the smallest integer 
such that $\oneplus \sumrab \plusone \not \leq \pi$.
These conditions are simpler than
those given in the weight function~(\ref{equation_general_weight_function})
for Theorem~\ref{theorem_mobius_sum_bottom_level_indecomposable} as,
by Corollary~\ref{corollary_add_2_for_op_or_po},
if 
$\sumrab \plusone \not \leq \pi$ then
$\oneplus \sumrab \not \leq \pi$ and vice-versa.
Furthermore, we will see that
this weight function
is only used when $\sigma \leq \sumra \leq \pi$.

We are now in a position to state our main theorem for 
this section.
In this theorem, we consider the contribution
to the \mob sum of each possible shape of 
some sum indecomposable $\alpha$.
There are five possible shapes, and, given that the 
expression for each shape is identical, 
we abuse notation slightly by writing our theorem
as a sum over the shapes, thus the
first sum in 
Theorem~\ref{theorem_mobius_sum_increasing_oscillations}
is over the possible shapes of $\alpha$,
where four of the shapes
have a parameter $k$.
For each shape, the limits on the interior sum
determine the minimum and maximum values
of $k$, using the summation variable $v$.  We use the 
notation $\alpha_v$ to represent the actual
permutation that has the shape $\alpha$,
where the parameter $k$ has been set to the value of $v$.
As an example, if $\alpha = \oneil \ildtwokb$,
and $v = 2$,
then $\alpha_v = \oneil \left( \nils{2}{\dtwo} \right) = 24153$.
 
\begin{theorem}
    \label{theorem_mobius_sum_increasing_oscillations}
    Let $\pi$ be an increasing oscillation,
    and let $\sigma \leq \pi$ be sum indecomposable.
    Then
    \begin{align*}
    \mobfn{\sigma}{\pi}
    & = 
    \sum_{\alpha \in \shapes} \
    \sum_{v=\mink(\sigma,\alpha)}^{\maxk(\alpha,\pi)}
    \mobfn{\sigma}{\alpha_{v}} 
    \weightosc{\sigma}{\alpha_{v}}{\pi}    
    \end{align*}
    where the first sum
    is over the possible shapes of a sum indecomposable
    permutation contained in an increasing oscillation,
    so 
    $\shapes = \{ \dtwo, \;$
    $\ildtwokp,\;$ 
    $\oneil \ildtwokb,\;$ 
    $\ildtwokb \ilone,\;$
    $\oneil \ildtwokb \ilone \}$.
\end{theorem}
\begin{proof}
    By Lemma~\ref{lemma_form_of_permutations_in_increasing_oscillations}
    the only sum-decomposable permutations contained
    in an increasing oscillation that contribute to the 
    \mob sum are $\familysum{\sumra}$,
    where $\alpha \in \shapes$.
    
    If we set $r=1$, then for each $\alpha$ in $\shapes$
    Lemma~\ref{lemma_inequalities_for_sigma} provides the
    smallest value of $k$ such that $\sigma \leq \alpha$.
    If there is no such value of $k$, then we use $\order{\pi}$,
    as the maximum value of $k$ must be smaller than this,
    and so the sum is empty.
    
    Again setting $r=1$, for each $\alpha$ in $\shapes$
    Lemma~\ref{lemma_inequalities_for_pi} provides the 
    maximum value of $k$ such that $\alpha \leq \pi$.
    If there is no value of $k$ that satisfies the 
    inequality, then we set $\maxk(\alpha,\pi) = 0$,
    thus forcing the sum to be empty.
    
    Thus the permutations $\alpha_v$ in the sum
    \[
    \sum_{\alpha \in \shapes} \
    \sum_{v=\mink(\sigma,\alpha)}^{\maxk(\alpha,\pi)}
    \] 
    are those that could contribute to the \mob sum,
    and for any $\alpha_v$ not included in the sum,
    $\familysum{\nsums{r}{\alpha_v}}$ has a zero contribution
    to the \mob sum for any $r$.
    
    Further, we can see from the construction method that
    any $\alpha_v$ included in the sum 
    has $\sigma \leq \nsums{r}{\alpha_v} \leq \pi$
    for at least one value of $r$, as if this was
    not the case, then we would have
    $\mink(\sigma,\alpha) > \maxk(\alpha,\pi)$, 
    and so the sum would be empty.
    
    We have therefore shown that the $\alpha_v$-s included 
    in the sum form a contributing set, and we could 
    therefore set $\contrib{\sigma}{\pi}$ to be those
    $\alpha_v$-s, and use
    Theorem~\ref{theorem_mobius_sum_bottom_level_indecomposable}.
    We now show that the increasing oscillation weight function
    $\weightosc{\sigma}{\alpha}{\pi}$ is equivalent
    to $\weightgen{\sigma}{\alpha}{\pi}$ as defined in the
    general case.
    
    By Corollary~\ref{corollary_add_2_for_op_or_po},
    if 
    $\sumrab \plusone \not \leq \pi$ then
    $\oneplus \sumrab \not \leq \pi$ and vice-versa,
    and so the condition for $\sumrab \plusone$ 
    also covers $\oneplus \sumrab$.
    As discussed above, 
    we know that there is at least one value of $r$ such that
    $\sigma \leq \sumra \leq \pi$, and so 
    $\weightosc{\sigma}{\alpha}{\pi}$ does not need to 
    include this condition.  
    Thus the increasing oscillation weight function
    $\weightosc{\sigma}{\alpha}{\pi}$ is equivalent
    to $\weightgen{\sigma}{\alpha}{\pi}$ as defined in the
    general case.
\end{proof}

\subsection{Example of Theorem~\ref{theorem_mobius_sum_increasing_oscillations}}
As an example of 
Theorem~\ref{theorem_mobius_sum_increasing_oscillations}
in action, we show how to determine
\[
\mobfn{3142}{315274968} = \mobfn{\nils{2}{\dtwo}}{\nils{4}{\dtwo} \ilone}.
\]
We start by considering each 
possible shape of $\alpha$,
setting $r=1$, and then using
the inequalities in
Lemmas~\ref{lemma_inequalities_for_pi}
and~\ref{lemma_inequalities_for_sigma}
to determine the minimum and maximum 
values of $k$.  This gives us

\begin{center}
\begin{tabular}{ccc}
    \toprule
    Shape of $\alpha$ & Minimum $k$ & Maximum $k$  \\ 
    \midrule
    $\dtwo$ & 1 & 1 \\
    $\ildtwok$ & 2 & 4 \\
    $\oneil \ildtwokb$ & 2 & 3 \\
    $\ildtwokb \ilone$ & 2 & 3 \\
    $\oneil \ildtwokb \ilone$ & 2 & 3 \\
    \bottomrule
\end{tabular}
\end{center}

For each shape of $\alpha$, and each value of $k$, we then
use the inequalities in
Lemma~\ref{lemma_inequalities_for_pi}
to determine the minimum value of $r$ such that
$\oneplus \sumrab \plusone \not \leq \pi$,
and we then calculate the weight using this
value of $r$.  This gives

\begin{center}
\begin{tabular}{ccr}
    \toprule
    $\alpha$ & $r$ & Weight \\ 
    \midrule
    $\dtwo$ & \multicolumn{2}{c}{No possibilities} \\
    $\nils{2}{\dtwo}$ & $2$ & $1$ \\
    $\nils{3}{\dtwo}$ & $1$ & $-1$ \\
    $\nils{4}{\dtwo}$ & $1$ & $-1$ \\
    $\oneil \left( \nils{2}{\dtwo} \right)$ & $1$ & $-1$ \\
    $\oneil \left( \nils{3}{\dtwo} \right)$ & $1$ & $-1$ \\
    $\left( \nils{2}{\dtwo} \right) \ilone$ & $2$ & $1$ \\
    $\left( \nils{3}{\dtwo} \right) \ilone$ & $1$ & $-1$ \\
    $\oneil \left( \nils{2}{\dtwo} \right) \ilone$ & $1$ & $-1$ \\    
    $\oneil \left( \nils{3}{\dtwo} \right) \ilone$ & $1$ & $-1$ \\
    \bottomrule
\end{tabular}
\end{center}

This leads to the following initial
expression:
\begin{align*}
\mobfn{\nils{2}{\dtwo}}{\nils{4}{\dtwo} \ilone}
= &
  \mobfn{\nils{2}{\dtwo}}{\nils{2}{\dtwo}}
- \mobfn{\nils{2}{\dtwo}}{\nils{3}{\dtwo}}
- \mobfn{\nils{2}{\dtwo}}{\nils{4}{\dtwo}} 
\\ &
- \mobfn{\nils{2}{\dtwo}}{\oneil \left( \nils{2}{\dtwo} \right)} 
- {} \mobfn{\nils{2}{\dtwo}}{\oneil \left( \nils{3}{\dtwo} \right)} 
\\ &
+ \mobfn{\nils{2}{\dtwo}}{\nils{2}{\dtwo} \ilone}
- \mobfn{\nils{2}{\dtwo}}{\nils{3}{\dtwo} \ilone}
\\ & 
- \mobfn{\nils{2}{\dtwo}}{\oneil \left( \nils{2}{\dtwo} \right) \ilone}
- \mobfn{\nils{2}{\dtwo}}{\oneil \left( \nils{3}{\dtwo} \right) \ilone}
\end{align*}

We know that $\mobfn{\nils{2}{\dtwo}}{\nils{2}{\dtwo}} = 1$,
and that 
\[
\mobfn{\nils{2}{\dtwo}}{\oneil \left( \nils{2}{\dtwo} \right)} =
\mobfn{\nils{2}{\dtwo}}{\nils{2}{\dtwo} \ilone} = 
-1.
\]  
Applying 
Theorem~\ref{theorem_mobius_sum_increasing_oscillations}
recursively to the other intervals eventually yields
\[
\mobfn{\nils{2}{\dtwo}}{\nils{4}{\dtwo} \ilone} = -6.
\]
\section{Concluding remarks}\label{section-concluding-remarks}

The results in~\cite{Burstein2011} provide
two recurrences to handle the case where 
$\pi$ is decomposable.  
This work handles the case where $\sigma$ is indecomposable.
It overlaps with~\cite{Burstein2011}
when $\sigma$ is indecomposable 
and $\pi$ is decomposable.
This leaves the case where 
$\sigma$ is decomposable and 
$\pi$ is indecomposable
for further investigation.

We can see that
by symmetry
$\mobfn{\sigma}{W_{n}} = \mobfn{\sigma^{-1}}{M_{n}}$.
If we consider the value of the 
principal \mob function, $\mobfn{1}{\pi}$,
where $\pi$ is either $W_n$ or $M_n$,
then it is simple to show that the absolute value
of the principal \mob function
is bounded above by $2^n$.
The weight function for increasing oscillations
can be $\pm 1$, and we can see no obvious reason why
there should not be two distinct values, $i$ and $j$,
with the same parity, 
such that the signs of $\mobfn{1}{W_i}$ and 
$\mobfn{1}{W_j}$ were different.
We have experimental evidence,
based on the values of $W_n$ and $M_n$
for 
$n=1 \ldots 
\text{2,000,000}  
$ 
that suggests that
$\mobfn{1}{W_{2n}} < 0$, and that 
$\mobfn{1}{W_{2n-1}} > 0$.

Figure~\ref{figure_increasing_oscillation_values}
is a log-log plot of the absolute values of $\mobfn{1}{W_{2n}}$
from $n=\text{8,000}$ to $n=\text{10,000}$.  
As can be seen, there seems to be some 
evidence of ``banding'', and we 
have confirmed that this pattern
continues for all values examined.
Examination of the values of 
$\mobfn{1}{W_{2n-1}$} reveals the same 
patterns.
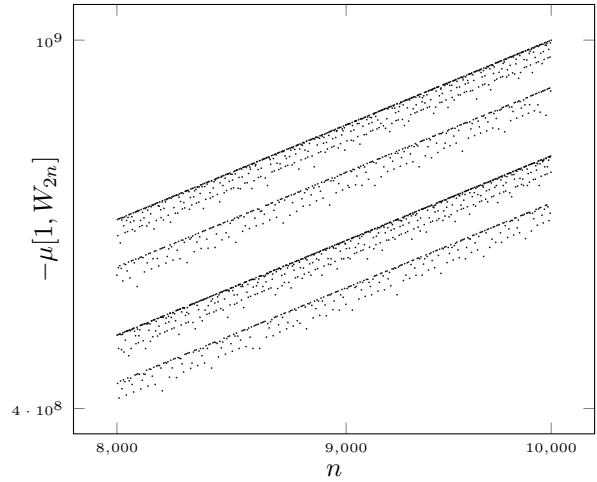
\begin{figure}[!htb]
    \centering
        \centering
        \input{MobiusFunctionOfIntervalsWithAnIndecomposableLowerBoundFinalFigures.tex}
        \caption{Log-Log plot of $\order{W_{2n}}$.}
        \label{figure_increasing_oscillation_values}
\end{figure}

Following discussions
at Permutation Patterns 2017,
V{\'{i}}t Jel{\'{i}}nek~\cite{Jelinek2017a} provided the following conjecture
(rephrased to reflect our notation).
\begin{conjecture}
    [{Jel{\'{i}}nek~\cite{Jelinek2017a}}]
    Let $M(n)$ denote the absolute value of the 
    \mob function $\mobfn{1}{W_n} = \mobfn{1}{M_n}$.  
    Then for $n > 50$ we have
    \begin{align*}
    M(2n) & = n^2 
    \Longleftrightarrow  
    \text{$n+1$ is prime and $n \equiv 0 \mymod 6$} 
    \\
    M(2n) & = n^2 - 1 
    \Longleftrightarrow  
    \text{$n+1$ is prime and $n \equiv 4 \mymod 6$} 
    \\
    M(2n+1) & = n^2 - n
    \Longleftrightarrow  
    \text{$n+1$ is prime and $n \equiv 0 \mymod 6$} 
    \\
    M(2n+1) & = n^2 - n - 1 
    \Longleftrightarrow  
    \text{$n+1$ is prime and $n \equiv 4 \mymod 6$} 
    \\
    \end{align*}    
    Further, Jel{\'{i}}nek notes that 
    there does not seem to be any other ``small'' constant $k$ such that $M(n) = (n^2-k)/4$ infinitely often.    
\end{conjecture}

We also have the following conjecture relating to the ``banding''
of the values.
\begin{conjecture}
    Let $M(n)$ denote the absolute value of the 
    \mob function $\mobfn{1}{W_n} = \mobfn{1}{M_n}$.  
    Let $E(n) = M(n) / (n^2)$, and
    let $O(n) = M(n)/(n^2 + n)$.
    Then, with $n \geq 1$, there exist constants 
    $0 < a < b < c < d < e < f < g < 1$ such that
    \begin{align*}
    E(12n + 10) & \in [ a , b ]  & O(12n + 11) & \in [ a , b ] \\
    E(12n +  2) & \in [ c , d ]  & O(12n +  3) & \in [ c , d ] \\
    E(12n +  6) & \in [ c , d ]  & O(12n +  7) & \in [ c , d ] \\
    E(12n +  4) & \in [ e , f ]  & O(12n +  5) & \in [ e , f ] \\
    E(12n +  8) & \in [ g , 1 ]  & O(12n +  9) & \in [ g , 1 ] \\
    E(12n     ) & \in [ g , 1 ]  & O(12n +  1) & \in [ g , 1 ] \\    
    \end{align*}
    Examining the first 1,500,000 values of 
    $\mobfn{1}{W_n}$
    gives the following
    estimates for the constants.
    \[
    \begin{array}{ccccccc}
        a & b & c & d & e & f & g \\
        0.615 & 0.680 &
        0.692 & 0.760 &
        0.821 & 0.896 &
        0.923
    \end{array}
    \]
\end{conjecture}

The \emph{complete nearly-layered} permutations 
are formed by interleaving descending
permutations.  Formally,
a complete nearly-layered permutation has the form
\[
\alpha_1 \interleave \alpha_2 \interleave
\ldots
\interleave \alpha_{k-1} \interleave \alpha_k
\]
where each $\alpha_i$ is a descending permutation,
with $\alpha_i > 1$ for $i = 2, \ldots ,k-1$.
If we set $\alpha_i = 21$ for $i = 2, \ldots ,k-1$,
and $\alpha_1, \alpha_k \in \{1,21 \}$,
then we obtain the increasing oscillations.

The computational approach taken 
for increasing oscillations
could, we think, be adapted to 
complete nearly-layered permutations. 
It is clear
that the equivalent of the inequalities
in 
Lemmas~\ref{lemma_inequalities_for_pi}
and~\ref{lemma_inequalities_for_sigma}
would be somewhat more complex than those found here,
but we believe that it should be possible
to define an algorithm that could determine
the \mob function for complete nearly-layered
permutations where the lower bound is sum indecomposable.

\paragraph*{Acknowledgments.}
We would like to thank the anonymous reviewers
for their feedback and suggestions
which led to numerous improvements to this paper.

\bibliographystyle{abbrv} 
\bibliography{../Bibliography}  

\end{document}

%% file: MobiusFunctionOfIntervalsWithAnIndecomposableLowerBoundFinalFigures.tex
    \begin{tikzpicture}[scale=1]
    \pgfplotsset{every tick label/.append style={font=\tiny}}
    \begin{loglogaxis}[
    log ticks with fixed point,
    xtick={8000,9000, 10000},
    ytick={40E6,100E6},
    yticklabels={$4 \cdot 10^8$, $10^9$},
    x label style={at={(axis description cs:0.5,-0.05)},anchor=north},
    y label style={at={(axis description cs:-0.01,.5)},anchor=south},
    xlabel=$n$, 
    ylabel=$-\mobfn{1}{W_{2n}}$
    ]
    \addplot[
    scatter,
    only marks,
    mark size=0.1,
    scatter/@pre marker code/.style={/tikz/mark size=0.1},
    scatter/@post marker code/.style={}
    ]
    table {
        8000	55729296
        8001	48016002
        8002	64006800
        8003	42572988
        8004	61504002
        8005	48064020
        8006	56779008
        8007	46442880
        8008	64128063
        8009	41053056
        8010	64160100
        8011	48136086
        8012	57063240
        8013	48160134
        8014	60399936
        8015	42830658
        8016	64256256
        8017	48074040
        8018	56677320
        8019	46304004
        8020	63940800
        8021	42020400
        8022	64335042
        8023	48100092
        8024	54941922
        8025	48304482
        8026	64295088
        8027	42959664
        8028	63024480
        8029	46028160
        8030	57319920
        8031	48376002
        8032	64436064
        8033	42766560
        8034	61965906
        8035	47409600
        8036	57221424
        8037	48449046
        8038	64609443
        8039	41357184
        8040	63853440
        8041	48497280
        8042	56319024
        8043	48521400
        8044	62120256
        8045	43151472
        8046	64356048
        8047	48569502
        8048	57577176
        8049	45613920
        8050	64786560
        8051	42849600
        8052	64834704
        8053	48642132
        8054	55356576
        8055	48514752
        8056	63576600
        8057	43125888
        8058	64931364
        8059	46442880
        8060	57749001
        8061	48678672
        8062	64459920
        8063	42462720
        8064	62429538
        8065	48747744
        8066	57835008
        8067	48811392
        8068	65092623
        8069	41675076
        8070	63797760
        8071	48859776
        8072	57470688
        8073	48480360
        8074	62196480
        8075	43475712
        8076	65181762
        8077	47934720
        8078	58007217
        8079	46995204
        8080	65286399
        8081	43540224
        8082	65296866
        8083	48956880
        8084	54183360
        8085	48739440
        8086	65383395
        8087	43604736
        8088	65415744
        8089	47116092
        8090	58050720
        8091	47935104
        8092	65480463
        8093	43540560
        8094	62894898
        8095	48653436
        8096	58266009
        8097	49175106
        8098	63854208
        8099	41986080
        8100	65610000
        8101	49223700
        8102	58299264
        8103	49247970
        8104	63050400
        8105	42905088
        8106	65151240
        8107	49296630
        8108	58216704
        8109	47349360
        8110	65772099
        8111	43605120
        8112	64268640
        8113	49369632
        8114	56185920
        8115	49393968
        8116	65869455
        8117	43540800
        8118	65777538
        8119	46443264
        8120	58611960
        8121	49413000
        8122	65966883
        8123	43994082
        8124	62988000
        8125	49344060
        8126	57471120
        8127	49537152
        8128	65519640
        8129	42297120
        8130	66065202
        8131	49447728
        8132	58785321
        8133	48594840
        8134	63518064
        8135	44120832
        8136	66178008
        8137	49368384
        8138	58872096
        8139	47272320
        8140	64909320
        8141	44093460
        8142	66063426
        8143	49735230
        8144	56600640
        8145	49759842
        8146	66357315
        8147	43347456
        8148	66310722
        8149	47815920
        8150	58052160
        8151	49833138
        8152	66385560
        8153	44317800
        8154	62528640
        8155	49882086
        8156	59132808
        8157	49906566
        8158	66512160
        8159	42463236
        8160	66585600
        8161	48516000
        8162	59219784
        8163	49682880
        8164	63854556
        8165	44450238
        8166	66683556
        8167	50028960
        8168	58097568
        8169	47894112
        8170	66748899
        8171	44514768
        8172	66230922
        8173	50078160
        8174	57015360
        8175	49095936
        8176	66175488
        8177	44508132
        8178	66879684
        8179	48170112
        8180	59475600
        8181	50200662
        8182	65578632
        8183	44232000
        8184	64301346
        8185	50249760
        8186	59568609
        8187	50173380
        8188	66858120
        8189	41800320
        8190	67076100
        8191	50323455
        8192	59655960
        8193	50173440
        8194	63924156
        8195	44777394
        8196	65805480
        8197	50397204
        8198	59743320
        8199	48377604
        8200	67217592
        8201	44843046
        8202	66875760
        8203	49441440
        8204	57439200
        8205	50078640
        8206	67258128
        8207	44784576
        8208	67371264
        8209	48524316
        8210	58384128
        8211	50569488
        8212	67399080
        8213	44941680
        8214	64683840
        8215	50311872
        8216	59502480
        8217	49610760
        8218	67535523
        8219	43237116
        8220	67568400
        8221	50692740
        8222	60093681
        8223	50716674
        8224	63579360
        8225	45105696
        8226	67480128
        8227	50174076
        8228	59824800
        8229	48761040
        8230	67732899
        8231	44250528
        8232	67765824
        8233	50743524
        8234	57844800
        8235	50795364
        8236	67831695
        8237	45237582
        8238	65927760
        8239	48875136
        8240	60308160
        8241	50638368
        8242	67930563
        8243	45302688
        8244	65015808
        8245	49759920
        8246	60445008
        8247	51013842
        8248	68011776
        8249	43196004
        8250	68011920
        8251	51063366
        8252	59295600
        8253	51088134
        8254	65016000
        8255	45411072
        8256	68032890
        8257	51137664
        8258	60621057
        8259	48101952
        8260	67665000
        8261	45344448
        8262	68260644
        8263	51211968
        8264	58052736
        8265	51236802
        8266	66934320
        8267	45282432
        8268	68359824
        8269	49236204
        8270	60797304
        8271	50864964
        8272	68425983
        8273	44702112
        8274	65722800
        8275	51360846
        8276	60815040
        8277	51385686
        8278	68288832
        8279	43790208
        8280	66764880
        8281	51399996
        8282	60469998
        8283	51313824
        8284	65882304
        8285	45766320
        8286	68657796
        8287	50423040
        8288	61061688
        8289	49474656
        8290	68724099
        8291	45832560
        8292	68757264
        8293	50796480
        8294	57500352
        8295	51417600
        8296	68823615
        8297	45898800
        8298	68818176
        8299	49587036
        8300	61239240
        8301	50630208
        8302	68603832
        8303	45963906
        8304	65652000
        8305	51733920
        8306	60953760
        8307	51693840
        8308	67616232
        8309	44190912
        8310	69056100
        8311	51808656
        8312	61202304
        8313	51833634
        8314	66360336
        8315	44785440
        8316	69155856
        8317	51883524
        8318	61460046
        8319	49539072
        8320	69195360
        8321	46028736
        8322	67724160
        8323	51958398
        8324	59088960
        8325	51883920
        8326	68750640
        8327	46231134
        8328	69355584
        8329	48765696
        8330	61682721
        8331	52058328
        8332	69012480
        8333	46297944
        8334	66679986
        8335	52108158
        8336	60511968
        8337	51702840
        8338	69449520
        8339	44508864
        8340	69363432
        8341	52149888
        8342	61854840
        8343	51141696
        8344	66840096
        8345	46153248
        8346	69415362
        8347	52258470
        8348	61322802
        8349	50191764
        8350	68301600
        8351	46443264
        8352	69755904
        8353	52333632
        8354	59558718
        8355	52358688
        8356	69800640
        8357	45614448
        8358	69443856
        8359	49760640
        8360	62127744
        8361	52391808
        8362	69923043
        8363	46443516
        8364	65790816
        8365	52453980
        8366	62217009
        8367	52509024
        8368	70023423
        8369	44785440
        8370	69479202
        8371	51087168
        8372	62306280
        8373	52557648
        8374	67306800
        8375	46765632
        8376	70157376
        8377	52609476
        8378	60954768
        8379	50554596
        8380	69899646
        8381	46443600
        8382	70241202
        8383	52707198
        8384	59601024
        8385	51660024
        8386	70324995
        8387	46899264
        8388	70358544
        8389	50675532
        8390	62574480
        8391	52810818
        8392	68415840
        8393	46967208
        8394	67503744
        8395	52861206
        8396	62663400
        8397	52247808
        8398	70474032
        8399	44232960
        8400	70486200
        8401	52936800
        8402	62753601
        8403	52523604
        8404	67764798
        8405	47101416
        8406	69220800
        8407	53012400
        8408	62843256
        8409	50857188
        8410	70310640
        8411	47168802
        8412	70727946
        8413	52005600
        8414	59712000
        8415	53112642
        8416	70633368
        8417	47134560
        8418	70862724
        8419	51038400
        8420	61737600
        8421	53189142
        8422	70930083
        8423	47024064
        8424	68127360
        8425	52800036
        8426	63090351
        8427	52150896
        8428	71031183
        8429	45476484
        8430	71064900
        8431	53076480
        8432	63202464
        8433	53340834
        8434	66896640
        8435	47272896
        8436	70142880
        8437	53391444
        8438	63211008
        8439	51280320
        8440	71099160
        8441	46526832
        8442	71267364
        8443	53467398
        8444	60848862
        8445	53456160
        8446	71334915
        8447	47180802
        8448	69659520
        8449	51099360
        8450	63472032
        8451	53568768
        8452	71418984
        8453	47641086
        8454	68416128
        8455	52524000
        8456	63562809
        8457	53644866
        8458	70947840
        8459	45780672
        8460	71571600
        8461	53695620
        8462	61922880
        8463	53619810
        8464	68776416
        8465	47604612
        8466	71673156
        8467	53698176
        8468	63743280
        8469	50175840
        8470	71700720
        8471	47843838
        8472	71721504
        8473	53698248
        8474	61277442
        8475	53553312
        8476	70376208
        8477	47910096
        8478	71853840
        8479	51750396
        8480	63395838
        8481	53949762
        8482	71696064
        8483	46996800
        8484	69101826
        8485	54000660
        8486	63856320
        8487	54026082
        8488	71620896
        8489	46126176
        8490	70611120
        8491	53629440
        8492	63925344
        8493	54043680
        8494	69264816
        8495	48101952
        8496	72095970
        8497	53049240
        8498	64195776
        8499	51833604
        8500	72249999
        8501	47894976
        8502	71687922
        8503	54229968
        8504	60458400
        8505	54255522
        8506	72317520
        8507	48251616
        8508	72365832
        8509	51999552
        8510	64377201
        8511	53077248
        8512	72454143
        8513	47895120
        8514	69176640
        8515	54383166
        8516	64243326
        8517	54408726
        8518	71077632
        8519	46444164
        8520	72590400
        8521	54459840
        8522	64558824
        8523	54485400
        8524	69108000
        8525	47411616
        8526	72692676
        8527	54182400
        8528	64649817
        8529	52380960
        8530	72559872
        8531	48516624
        8532	71152224
        8533	54423996
        8534	62152830
        8535	54182400
        8536	72863295
        8537	48592584
        8538	72897444
        8539	51419520
        8540	64437120
        8541	54715782
        8542	72965763
        8543	48654846
        8544	70082658
        8545	54767040
        8546	63028800
        8547	54792672
        8548	73051176
        8549	46651680
        8550	72850242
        8551	54843936
        8552	65014200
        8553	53409312
        8554	70143804
        8555	48655200
        8556	73164168
        8557	54467436
        8558	65104848
        8559	52745604
        8560	71780280
        8561	48866166
        8562	73307844
        8563	54997998
        8564	62590560
        8565	55023702
        8566	72943248
        8567	47771136
        8568	72561600
        8569	52873404
        8570	65288160
        8571	55100808
        8572	73479183
        8573	49003248
        8574	69137040
        8575	55140096
        8576	65379552
        8577	55177986
        8578	73443216
        8579	46444800
        8580	73616400
        8581	54103248
        8582	65471121
        8583	55149696
        8584	70489596
        8585	49123152
        8586	73642320
        8587	55149696
        8588	64225728
        8589	53120496
        8590	73165398
        8591	49207746
        8592	73386720
        8593	55384032
        8594	63028320
        8595	54279504
        8596	73891215
        8597	49277982
        8598	73925604
        8599	53215680
        8600	65699040
        8601	54736008
        8602	72486384
        8603	49345968
        8604	71070402
        8605	55210320
        8606	65653200
        8607	55563906
        8608	74097663
        8609	46444800
        8610	74114352
        8611	55616286
        8612	65308320
        8613	55615968
        8614	71235696
        8615	49484190
        8616	72722520
        8617	55633320
        8618	65405184
        8619	53492004
        8620	74249280
        8621	49553304
        8622	74338884
        8623	54183360
        8624	63351198
        8625	55642104
        8626	74407875
        8627	49622418
        8628	74442384
        8629	53616396
        8630	64852416
        8631	55536960
        8632	74494464
        8633	49691526
        8634	70973760
        8635	55730304
        8636	66297609
        8637	54811680
        8638	74586096
        8639	47771136
        8640	74649600
        8641	55935708
        8642	66339504
        8643	56030400
        8644	69662592
        8645	49415604
        8646	74753316
        8647	55951452
        8648	66413160
        8649	53863524
        8650	74776800
        8651	48876768
        8652	74598786
        8653	56160132
        8654	63912960
        8655	56185920
        8656	74308440
        8657	49636800
        8658	73433232
        8659	53989632
        8660	66666600
        8661	56238000
        8662	75030243
        8663	49899168
        8664	72065058
        8665	55167624
        8666	66753288
        8667	55875684
        8668	75134223
        8669	47937276
        8670	74496576
        8671	56393280
        8672	65469936
        8673	56420034
        8674	72230958
        8675	50175360
        8676	75272976
        8677	56472084
        8678	66390798
        8679	53078400
        8680	75342399
        8681	50245608
        8682	75168864
        8683	56214240
        8684	64355328
        8685	56540400
        8686	73640448
        8687	50313600
        8688	75481344
        8689	53907360
        8690	67129281
        8691	56600964
        8692	75550863
        8693	49263984
        8694	72479808
        8695	56706576
        8696	66823776
        8697	56732706
        8698	75655203
        8699	48379524
        8700	73530240
        8701	56626992
        8702	67314744
        8703	56614914
        8704	72732000
        8705	50523798
        8706	75794436
        8707	55703280
        8708	67407657
        8709	54280512
        8710	75784800
        8711	50175840
        8712	75898944
        8713	56941632
        8714	63471936
        8715	56967768
        8716	75825288
        8717	50663184
        8718	76003524
        8719	54736128
        8720	67174272
        8721	55876128
        8722	74980800
        8723	50732880
        8724	73066080
        8725	57098580
        8726	67686609
        8727	57124722
        8728	74497248
        8729	48766464
        8730	76212900
        8731	57119328
        8732	67726398
        8733	56730960
        8734	73234224
        8735	49540608
        8736	76317696
        8737	57057276
        8738	67872840
        8739	54736704
        8740	76387599
        8741	50866800
        8742	74865024
        8743	57334368
        8744	64688004
        8745	57360642
        8746	76492515
        8747	51012504
        8748	76075776
        8749	53994000
        8750	68059440
        8751	57439200
        8752	76597503
        8753	51082488
        8754	73309440
        8755	57016080
        8756	66759984
        8757	57447600
        8758	76490640
        8759	49098240
        8760	76737600
        8761	57230208
        8762	68113584
        8763	56422080
        8764	73738176
        8765	51222456
        8766	76209042
        8767	57646590
        8768	68279328
        8769	55370208
        8770	75343128
        8771	51088128
        8772	76867560
        8773	57689436
        8774	65309760
        8775	57754818
        8776	76996920
        8777	49762080
        8778	77053284
        8779	55496352
        8780	68526921
        8781	57833862
        8782	77123523
        8783	51419520
        8784	72564000
        8785	57775956
        8786	68532798
        8787	57570864
        8788	76293756
        8789	49443780
        8790	77238834
        8791	56781504
        8792	68714304
        8793	57992034
        8794	74243856
        8795	51573792
        8796	77155848
        8797	58015980
        8798	67405008
        8799	55289604
        8800	76982880
        8801	51642360
        8802	77475204
        8803	58052400
        8804	66147678
        8805	56726784
        8806	77545635
        8807	51714336
        8808	77434458
        8809	55876476
        8810	68418240
        8811	58229688
        8812	76068744
        8813	51475200
        8814	74497920
        8815	58052736
        8816	69090009
        8817	58309026
        8818	77757123
        8819	48767040
        8820	77792400
        8821	57879996
        8822	68943102
        8823	58388370
        8824	74750718
        8825	51926280
        8826	75852288
        8827	58441350
        8828	69272496
        8829	56130480
        8830	77968899
        8831	51898878
        8832	77346720
        8833	57327480
        8834	66346560
        8835	58521018
        8836	78074895
        8837	52067400
        8838	78110244
        8839	55732224
        8840	68050080
        8841	58626882
        8842	78123024
        8843	51696480
        8844	74981760
        8845	58679940
        8846	69560952
        8847	57500352
        8848	78287103
        8849	50107356
        8850	78292242
        8851	58759566
        8852	69242784
        8853	58622688
        8854	72980160
        8855	52248960
        8856	78158082
        8857	58802184
        8858	69749856
        8859	56512356
        8860	78499599
        8861	51282000
        8862	78535044
        8863	58918272
        8864	67050816
        8865	58053600
        8866	78605955
        8867	52421616
        8868	77036400
        8869	56640204
        8870	69938961
        8871	59025378
        8872	78494904
        8873	52249344
        8874	75586002
        8875	57873504
        8876	69454878
        8877	58992384
        8878	78353616
        8879	50425344
        8880	78836442
        8881	59158560
        8882	68667696
        8883	59185200
        8884	75771264
        8885	52634718
        8886	78960996
        8887	58743996
        8888	70223097
        8889	55732320
        8890	78759360
        8891	52249536
        8892	79067664
        8893	59318532
        8894	67506942
        8895	59342208
        8896	77400000
        8897	52776984
        8898	78521442
        8899	57016956
        8900	70242480
        8901	59425302
        8902	79151184
        8903	51752064
        8904	75659520
        8905	59451840
        8906	70507809
        8907	59296644
        8908	79326600
        8909	50384160
        8910	77538384
        8911	59558718
        8912	70602840
        8913	59585634
        8914	76283856
        8915	52990674
        8916	79436160
        8917	58077936
        8918	70697880
        8919	57279552
        8920	78910200
        8921	53062086
        8922	79602084
        8923	59600640
        8924	66347520
        8925	59746182
        8926	79654848
        8927	53078400
        8928	79709184
        8929	57224736
        8930	70468416
        8931	58054080
        8932	79780623
        8933	53204928
        8934	76626546
        8935	59880096
        8936	70983000
        8937	59866560
        8938	78259632
        8939	51144060
        8940	79923600
        8941	59752764
        8942	70491600
        8943	59601024
        8944	76798176
        8945	52249680
        8946	79879890
        8947	60041070
        8948	70975008
        8949	57664884
        8950	80102499
        8951	53419200
        8952	78504984
        8953	59582160
        8954	68419296
        8955	60148488
        8956	79708512
        8957	53491182
        8958	79886400
        8959	56616960
        8960	71274336
        8961	60229122
        8962	80317443
        8963	53555184
        8964	76502880
        8965	60282900
        8966	69985440
        8967	60126048
        8968	80425023
        8969	51089280
        8970	80460900
        8971	60363606
        8972	71556624
        8973	59159040
        8974	77313678
        8975	53079036
        8976	80530362
        8977	60431052
        8978	71596800
        8979	58053636
        8980	78996840
        8981	53778024
        8982	80200080
        8983	60525168
        8984	68880030
        8985	60552162
        8986	79818480
        8987	52747296
        8988	80766402
        8989	58054080
        8990	71791920
        8991	60632322
        8992	80425662
        8993	53922006
        8994	76075008
        8995	60326592
        8996	71940009
        8997	60212640
        8998	80964003
        8999	51835200
        9000	81000000
        9001	59528808
        9002	72036000
        9003	60795000
        9004	77832000
        9005	53907984
        9006	81108036
        9007	60848862
        9008	69664320
        9009	58220040
        9010	81180099
        9011	54138000
        9012	81216144
        9013	60930132
        9014	69340800
        9015	59601696
        9016	81267480
        9017	54208296
        9018	81227442
        9019	58054392
        9020	72241920
        9021	60677328
        9022	79737504
        9023	54257922
        9024	77962320
        9025	61092480
        9026	72150144
        9027	61058880
        9028	81504783
        9029	51089472
        9030	80868402
        9031	61173696
        9032	72516921
        9033	61200834
        9034	77884992
        9035	54426000
        9036	79985160
        9037	61255044
        9038	72472398
        9039	58827780
        9040	81721599
        9041	54046404
        9042	81757764
        9043	59712768
        9044	69787872
        9045	61363542
        9046	81811728
        9047	54184704
        9048	81866304
        9049	58960800
        9050	71321520
        9051	61369920
        9052	81262920
        9053	54643704
        9054	78698802
        9055	61498560
        9056	72902808
        9057	60271800
        9058	82047363
        9059	52526400
        9060	81269760
        9061	61463244
        9062	72772128
        9063	61093920
        9064	77208768
        9065	54788838
        9066	82192356
        9067	61662390
        9068	73096137
        9069	59223408
        9070	82225920
        9071	53742528
        9072	82255320
        9073	61378848
        9074	69687198
        9075	61771248
        9076	82275648
        9077	54737412
        9078	80730432
        9079	59353020
        9080	73289664
        9081	61681032
        9082	82396560
        9083	55006560
        9084	79059552
        9085	60127440
        9086	72951936
        9087	61922562
        9088	82566240
        9089	52872000
        9090	82628100
        9091	61989246
        9092	71985024
        9093	62016534
        9094	79115004
        9095	55151712
        9096	82054842
        9097	62071104
        9098	73579968
        9099	58055040
        9100	82581192
        9101	55152000
        9102	82846404
        9103	62152830
        9104	70732320
        9105	62104224
        9106	81229200
        9107	54737760
        9108	82955664
        9109	59746956
        9110	73774800
        9111	62033664
        9112	82538400
        9113	54185040
        9114	79745298
        9115	62261472
        9116	73871952
        9117	62309760
        9118	82452240
        9119	53079552
        9120	81479160
        9121	62399040
        9122	73969281
        9123	62426400
        9124	79905600
        9125	55188432
        9126	83283876
        9127	61204896
        9128	73808766
        9129	59505768
        9130	83199600
        9131	55589442
        9132	83393424
        9133	62563332
        9134	69664896
        9135	62590560
        9136	83466495
        9137	55662582
        9138	82721088
        9139	60140928
        9140	73647840
        9141	61394928
        9142	83525280
        9143	55732320
        9144	80165280
        9145	62510460
        9146	74359008
        9147	62755272
        9148	81980232
        9149	53563200
        9150	83722500
        9151	61923840
        9152	74450880
        9153	62717688
        9154	80446800
        9155	54737856
        9156	83832336
        9157	62717760
        9158	74499600
        9159	60403584
        9160	83905599
        9161	55955184
        9162	81271680
        9163	62890464
        9164	71212416
        9165	63002502
        9166	83997408
        9167	56027202
        9168	84019938
        9169	59298720
        9170	74749704
        9171	63085008
        9172	84125583
        9173	55636080
        9174	80798160
        9175	63028800
        9176	72981216
        9177	62793984
        9178	84212832
        9179	53743104
        9180	84272400
        9181	63222660
        9182	74945520
        9183	61923840
        9184	80304156
        9185	56248920
        9186	84382596
        9187	63305310
        9188	75043440
        9189	60800976
        9190	82238400
        9191	56322078
        9192	84392226
        9193	63388032
        9194	72137664
        9195	62718480
        9196	84274560
        9197	55235520
        9198	84603204
        9199	60818556
        9200	75239640
        9201	63458808
        9202	84676803
        9203	56119872
        9204	79668960
        9205	63553620
        9206	74638800
        9207	63581202
        9208	84787263
        9209	54281952
        9210	84798000
        9211	62310864
        9212	75370464
        9213	63443520
        9214	81271296
        9215	56616960
        9216	84433248
        9217	63193236
        9218	73993968
        9219	61198404
        9220	85008399
        9221	56603460
        9222	84884802
        9223	63802368
        9224	72566400
        9225	62528424
        9226	85119075
        9227	56764416
        9228	84453642
        9229	60957120
        9230	75467520
        9231	63912960
        9232	83492664
        9233	56681208
        9234	81859026
        9235	63968526
        9236	75829608
        9237	63927120
        9238	85340643
        9239	53080320
        9240	85377600
        9241	64051680
        9242	75467808
        9243	64079400
        9244	81992400
        9245	56866128
        9246	83746080
        9247	63913470
        9248	76026777
        9249	61552800
        9250	84756558
        9251	57059328
        9252	85362984
        9253	62908560
        9254	73082238
        9255	63858432
        9256	85673535
        9257	57134184
        9258	85670034
        9259	61730592
        9260	74670120
        9261	63798000
        9262	85756944
        9263	57206784
        9264	82100736
        9265	64341756
        9266	76323009
        9267	63098640
        9268	85140000
        9269	54986976
        9270	85911840
        9271	64272960
        9272	75790398
        9273	64496034
        9274	80857920
        9275	57356514
        9276	86044176
        9277	64551684
        9278	76520760
        9279	61924356
        9280	86118399
        9281	55733760
        9282	86155524
        9283	64100400
        9284	73556832
        9285	64663062
        9286	86166000
        9287	57474144
        9288	84508632
        9289	62131452
        9290	76504176
        9291	64618404
        9292	86341263
        9293	57579408
        9294	81756480
        9295	63471936
        9296	76817952
        9297	64830306
        9298	86154432
        9299	55291200
        9300	86468202
        9301	64886100
        9302	75348624
        9303	64913970
        9304	83104800
        9305	57225840
        9306	86548962
        9307	64584096
        9308	76918686
        9309	60957792
        9310	86676099
        9311	57796608
        9312	86689944
        9313	65053632
        9314	73893600
        9315	64853508
        9316	84317640
        9317	57877182
        9318	86825124
        9319	62532348
        9320	76757856
        9321	65136552
        9322	86899683
        9323	56728512
        9324	83461920
        9325	65221380
        9326	77314608
        9327	64688004
        9328	86771160
        9329	55705764
        9330	85140720
        9331	65305326
        9332	77126400
        9333	64946928
        9334	83641584
        9335	58100670
        9336	87160896
        9337	63859488
        9338	76873200
        9339	62801892
        9340	87235599
        9341	58173840
        9342	87272964
        9343	65461248
        9344	72982272
        9345	65501442
        9346	86832048
        9347	58055040
        9348	87385104
        9349	62201592
        9350	77712984
        9351	64245984
        9352	87418584
        9353	58325286
        9354	84000402
        9355	65641686
        9356	77812809
        9357	65669766
        9358	85785000
        9359	55733760
        9360	86660640
        9361	65655000
        9362	77912640
        9363	65754000
        9364	84180096
        9365	57281328
        9366	87075072
        9367	65810160
        9368	78011928
        9369	63206208
        9370	87796899
        9371	58055196
        9372	85528224
        9373	65853648
        9374	74990001
        9375	65921922
        9376	87909375
        9377	58624800
        9378	87927618
        9379	62035776
        9380	78162654
        9381	66007062
        9382	87295920
        9383	58386948
        9384	84540066
        9385	65491632
        9386	76712544
        9387	66091512
        9388	88080960
        9389	56424576
        9390	88172100
        9391	66147678
        9392	78324000
        9393	64274400
        9394	84720336
        9395	58780512
        9396	88284816
        9397	66180240
        9398	78048000
        9399	63583044
        9400	86246784
        9401	58925448
        9402	88397604
        9403	66316998
        9404	74640960
        9405	66345222
        9406	88305888
        9407	57797376
        9408	88501248
        9409	63747036
        9410	78713601
        9411	66035520
        9412	88585743
        9413	59075784
        9414	83345856
        9415	65931996
        9416	78757056
        9417	66284160
        9418	88698723
        9419	56784000
        9420	88736400
        9421	65213568
        9422	78913872
        9423	66346560
        9424	84658560
        9425	59226678
        9426	88116642
        9427	66655950
        9428	77403648
        9429	63859848
        9430	88924899
        9431	59298720
        9432	88962624
        9433	66708924
        9434	75635712
        9435	65407104
        9436	89038095
        9437	58539360
        9438	89075844
        9439	64135932
        9440	79216224
        9441	66854082
        9442	87076080
        9443	59453040
        9444	85624578
        9445	66910740
        9446	79263888
        9447	66939042
        9448	88528440
        9449	55982880
        9450	88775280
        9451	66761088
        9452	79264062
        9453	66942288
        9454	85694400
        9455	59602818
        9456	87591168
        9457	67080864
        9458	79518600
        9459	63860160
        9460	89491599
        9461	59506416
        9462	89529444
        9463	65407680
        9464	76436640
        9465	67194402
        9466	89605155
        9467	59754864
        9468	89333634
        9469	64562604
        9470	77404800
        9471	67230720
        9472	89718783
        9473	59831448
        9474	86168880
        9475	67203312
        9476	79350192
        9477	65991120
        9478	89832483
        9479	57502848
        9480	89622072
        9481	66864996
        9482	79821504
        9483	67450200
        9484	84589920
        9485	59713920
        9486	89949810
        9487	67506942
        9488	80023896
        9489	64440576
        9490	90060099
        9491	58829568
        9492	89354922
        9493	67555596
        9494	76920480
        9495	67620978
        9496	90174015
        9497	60134982
        9498	88181808
        9499	64792800
        9500	80226441
        9501	67706502
        9502	89397504
        9503	59713920
        9504	86716002
        9505	66374784
        9506	80327808
        9507	67792032
        9508	90335232
        9509	57876036
        9510	90440100
        9511	67728636
        9512	78786000
        9513	67849320
        9514	86178876
        9515	59990400
        9516	90459720
        9517	67934724
        9518	80305488
        9519	63694848
        9520	90630399
        9521	60325320
        9522	90649506
        9523	68020398
        9524	77404320
        9525	67487040
        9526	88895040
        9527	60515136
        9528	90246816
        9529	65383356
        9530	80733312
        9531	68134728
        9532	90859023
        9533	59354832
        9534	87264306
        9535	68188926
        9536	79891842
        9537	68031000
        9538	90973443
        9539	58221696
        9540	89011104
        9541	67873968
        9542	80937360
        9543	68306370
        9544	87270300
        9545	60667488
        9546	91126116
        9547	66348000
        9548	81039120
        9549	65656404
        9550	91202499
        9551	60819264
        9552	91185282
        9553	68212476
        9554	75858048
        9555	68478288
        9556	91064808
        9557	60879816
        9558	90587640
        9559	65794620
        9560	81243000
        9561	67166088
        9562	91409760
        9563	60973602
        9564	87814338
        9565	68621700
        9566	81344952
        9567	68116224
        9568	89680632
        9569	58056312
        9570	91268802
        9571	68707806
        9572	81447081
        9573	68736534
        9574	87997518
        9575	59714496
        9576	91671840
        9577	68793984
        9578	81457200
        9579	66071076
        9580	90462240
        9581	61203408
        9582	89878032
        9583	68880030
        9584	78372000
        9585	68908962
        9586	91891395
        9587	61040892
        9588	91878696
        9589	64855680
        9590	81595536
        9591	68420160
        9592	91971360
        9593	60958080
        9594	88113600
        9595	69052806
        9596	80186688
        9597	69081606
        9598	92011920
        9599	58977282
        9600	92160000
        9601	69139200
        9602	81273600
        9603	67757592
        9604	88238076
        9605	61510398
        9606	91730448
        9607	69225600
        9608	82060857
        9609	66416880
        9610	90469680
        9611	61579584
        9612	92390544
        9613	68420880
        9614	78878718
        9615	69340800
        9616	92437848
        9617	60405408
        9618	92505924
        9619	66182400
        9620	82265904
        9621	69187008
        9622	92582883
        9623	61740798
        9624	86388000
        9625	69485280
        9626	82368609
        9627	69421860
        9628	92698383
        9629	59341344
        9630	92736900
        9631	68116608
        9632	81757728
        9633	69600834
        9634	89011836
        9635	61372800
        9636	92677290
        9637	69629040
        9638	80611200
        9639	66900480
        9640	92832600
        9641	61972326
        9642	92968164
        9643	69745398
        9644	79371744
        9645	67923648
        9646	92277840
        9647	62035776
        9648	93083904
        9649	67038720
        9650	82779840
        9651	69663888
        9652	91259952
        9653	62126688
        9654	89474802
        9655	69663996
        9656	82724544
        9657	69370440
        9658	92726256
        9659	58388352
        9660	93315600
        9661	70006020
        9662	82985841
        9663	70032000
        9664	89660256
        9665	62279352
        9666	91527120
        9667	70092990
        9668	82402398
        9669	67318608
        9670	93250512
        9671	61925760
        9672	93224706
        9673	68748840
        9674	79823520
        9675	70147524
        9676	93624975
        9677	62435982
        9678	93663684
        9679	66901116
        9680	81596640
        9681	70258008
        9682	93564240
        9683	62512608
        9684	89496192
        9685	70268460
        9686	83398608
        9687	68945856
        9688	93857343
        9689	59715072
        9690	93121602
        9691	70441446
        9692	83501352
        9693	70415280
        9694	88377408
        9695	62662398
        9696	94012416
        9697	70111584
        9698	83553600
        9699	67729920
        9700	94070592
        9701	60958800
        9702	94030560
        9703	70615968
        9704	80362398
        9705	70510440
        9706	93881280
        9707	62823618
        9708	92052288
        9709	67876716
        9710	83305440
        9711	70732320
        9712	93544920
        9713	62901366
        9714	90463296
        9715	69346704
        9716	83853120
        9717	70776384
        9718	94439523
        9719	60460992
        9720	94478400
        9721	70878240
        9722	82306224
        9723	69665280
        9724	90776478
        9725	63056880
        9726	94571682
        9727	70769664
        9728	83927184
        9729	66763296
        9730	94602960
        9731	63134640
        9732	94711824
        9733	70976880
        9734	80169924
        9735	71082498
        9736	92301408
        9737	63212400
        9738	94828644
        9739	68296608
        9740	84037758
        9741	71170182
        9742	94906563
        9743	61926144
        9744	91150818
        9745	70640436
        9746	84201768
        9747	71257872
        9748	95023503
        9749	60475200
        9750	93122568
        9751	71156796
        9752	84539001
        9753	71345634
        9754	91338000
        9755	63445680
        9756	94394442
        9757	69665280
        9758	84643056
        9759	68559360
        9760	95204760
        9761	63524568
        9762	94734000
        9763	71491998
        9764	79617600
        9765	71322624
        9766	95374755
        9767	63031680
        9768	95413824
        9769	68718204
        9770	84851361
        9771	70148448
        9772	95378304
        9773	63678960
        9774	91224324
        9775	71212416
        9776	84955608
        9777	71697186
        9778	92882160
        9779	61207488
        9780	95648400
        9781	71726688
        9782	85059864
        9783	71785170
        9784	91638432
        9785	62534208
        9786	95765796
        9787	71843910
        9788	84660000
        9789	68421768
        9790	95844099
        9791	63694848
        9792	93928824
        9793	71901204
        9794	81859902
        9795	71875440
        9796	95942400
        9797	63860724
        9798	95942562
        9799	67730880
        9800	84668760
        9801	71539104
        9802	96079203
        9803	63860832
        9804	92177280
        9805	72108420
        9806	83734608
        9807	72137664
        9808	95864832
        9809	61580736
        9810	96236100
        9811	71599440
        9812	85582440
        9813	70753200
        9814	91915200
        9815	64228992
        9816	96353856
        9817	72285024
        9818	85687080
        9819	69423204
        9820	94263840
        9821	64307886
        9822	95369040
        9823	72372672
        9824	82357602
        9825	72152898
        9826	96449520
        9827	62700480
        9828	96589584
        9829	69564876
        9830	85788414
        9831	72491298
        9832	96668223
        9833	63930084
        9834	90948480
        9835	72550326
        9836	86001552
        9837	72579846
        9838	96786243
        9839	61926396
        9840	96253920
        9841	70909344
        9842	85807104
        9843	72525204
        9844	92260956
        9845	64622376
        9846	96889968
        9847	72727440
        9848	84435120
        9849	69846564
        9850	97022499
        9851	64701282
        9852	97030986
        9853	72385488
        9854	82850688
        9855	70771200
        9856	97140735
        9857	64690080
        9858	97180164
        9859	69665784
        9860	86180256
        9861	72934422
        9862	95276544
        9863	64855680
        9864	93409698
        9865	72993600
        9866	85145280
        9867	73023192
        9868	97353480
        9869	61042368
        9870	97416900
        9871	73082238
        9872	86632224
        9873	73112034
        9874	93584400
        9875	65016912
        9876	95203968
        9877	72567036
        9878	86663808
        9879	69666048
        9880	97555200
        9881	65078640
        9882	97653924
        9883	71765760
        9884	83371230
        9885	73289862
        9886	97732995
        9887	65169024
        9888	96751200
        9889	70246176
        9890	85172256
        9891	73378848
        9892	97273920
        9893	65021952
        9894	93978738
        9895	73438176
        9896	87054009
        9897	71962800
        9898	97699680
        9899	62203200
        9900	98010000
        9901	73527300
        9902	87159600
        9903	73556832
        9904	92247840
        9905	65022048
        9906	98128836
        9907	73616430
        9908	87264648
        9909	70701840
        9910	97027836
        9911	64137792
        9912	98062362
        9913	73705632
        9914	83878080
        9915	73665504
        9916	98280840
        9917	65312352
        9918	95784768
        9919	70771200
        9920	87476760
        9921	73172004
        9922	98446083
        9923	65650482
        9924	94548960
        9925	72377424
        9926	87582552
        9927	73645056
        9928	98565183
        9929	63101280
        9930	98604900
        9931	73534560
        9932	85145760
        9933	74003334
        9934	94740144
        9935	65685312
        9936	98451288
        9937	74062944
        9938	87794496
        9939	69666240
        9940	98803599
        9941	65888928
        9942	98813520
        9943	73534716
        9944	83598336
        9945	74182242
        9946	96791856
        9947	65968416
        9948	98962704
        9949	71272080
        9950	87907920
        9951	74271042
        9952	98969184
        9953	64690704
        9954	94334400
        9955	74121840
        9956	88112808
        9957	73921584
        9958	98974656
        9959	63474564
        9960	97179480
        9961	74163324
        9962	88166880
        9963	74391204
        9964	95313216
        9965	65658000
        9966	99321156
        9967	72982272
        9968	88325337
        9969	71560608
        9970	98786064
        9971	66286368
        9972	99440784
        9973	74600532
        9974	82935360
        9975	74501952
        9976	98699640
        9977	66366984
        9978	99216834
        9979	71704032
        9980	88538064
        9981	72982800
        9982	99614064
        9983	66054144
        9984	95696226
        9985	74780160
        9986	88644609
        9987	74191284
        9988	97726632
        9989	63820224
        9990	99780480
        9991	74870016
        9992	88751160
        9993	74692008
        9994	95888016
        9995	65022720
        9996	99330048
        9997	74960004
        9998	88116000
        9999	71992002
        10000	99976320        
    };
    \end{loglogaxis}
    \end{tikzpicture}